\documentclass[12pt,a4paper,twoside]{article}
\usepackage{amssymb,amsmath,amsfonts,amstext,amsthm}
\usepackage{bbm}
\usepackage{graphicx}
\usepackage{authblk}
\usepackage[colorlinks=true,breaklinks=true]{hyperref}
\usepackage[bbgreekl]{mathbbol}
\usepackage{multirow}
\usepackage{url}
\usepackage{abstract}
\usepackage{tikz}
\usepackage{tikz-cd}
\usetikzlibrary{arrows,matrix,shapes,positioning}
\usetikzlibrary{calc,patterns,shapes.geometric}
\usetikzlibrary{decorations.markings,decorations.pathmorphing,decorations.pathreplacing}
\usetikzlibrary{decorations.text,decorations.shapes}

\usepackage[english]{babel}
\usepackage{epstopdf}
\usepackage{subcaption}

\let\div\undefined
\DeclareMathOperator{\div}{div}
\DeclareMathOperator{\curl}{curl}
\DeclareMathOperator{\grad}{grad}

\newcommand{\du}{\,\mathrm{d}}
\newcommand{\R}{\mathbb{R}}
\newcommand{\N}{\mathbb{N}}
\newcommand{\xb}{{\mathbf{x}}}
\newcommand{\vb}{{\mathbf{v}}}
\newcommand{\Xb}{{\mathbf{X}}}
\newcommand{\Vb}{{\mathbf{V}}}
\newcommand{\ub}{\tilde{\mathbf{u}}}
\newcommand{\eb}{{\mathbf{e}}}

\newcommand{\Eb}{{\mathbf{E}}}
\newcommand{\Bb}{{\mathbf{B}}}
\newcommand{\Db}{{\mathbf{D}}}
\newcommand{\Hb}{{\mathbf{H}}}
\newcommand{\Ab}{{\mathbf{A}}}
\newcommand{\bb}{{\mathbf{b}}}
\newcommand{\bt}{\mathbf{t}}
\newcommand{\bn}{\mathbf{n}}
\newcommand{\rhob}{{\boldsymbol{\rho}}}

\newcommand{\xib}{{\boldsymbol{\xi}}}
\newcommand{\Xib}{{\boldsymbol{\Xi}}}

\newcommand{\tub}{\tilde{\mathbf{u}}}
\newcommand{\teb}{{\tilde{\mathbf{e}}}}
\newcommand{\te}{{\tilde{e}}}

\newcommand{\tEb}{{\tilde{\mathbf{E}}}}
\newcommand{\tBb}{{\tilde{\mathbf{B}}}}

\newcommand{\tPhi}{{\tilde{\Phi}}}

\newcommand{\tbb}{{\tilde{\mathbf{b}}}}
\newcommand{\tb}{{\tilde{b}}}
\newcommand{\tB}{{\tilde{B}}}

\newcommand{\trhob}{{\tilde{\boldsymbol{\rho}}}}

\newcommand{\tf}{\tilde{f}}
\newcommand{\tJb}{\tilde{\mathbf{J}}}
\newcommand{\Jb}{\mathbf{J}}
\newcommand{\BB}{{\ensuremath{\mathbb{B}}}}
\newcommand{\MM}{\ensuremath{\mathbb{M}}}

\newcommand{\C}{\ensuremath{\mathsf{C}}}
\newcommand{\D}{\ensuremath{\mathsf{D}}}
\newcommand{\G}{\ensuremath{\mathsf{G}}}
\newcommand{\tBB}{\ensuremath\tilde{{\mathbb{B}}}}

\newcommand{\tM}{\ensuremath\tilde{\mathsf{M}}}
\newcommand{\NN}{\ensuremath{\mathbb{N}}}
\newcommand{\tom}{\ensuremath{\tilde{\Omega}}}
\newcommand{\tLa}{\ensuremath{\tilde{\Lambda}}}
\newcommand{\tLam}{\ensuremath{\tilde{\Lambda}}}
\newcommand{\tLab}{\ensuremath{\tilde{\boldsymbol{\Lambda}}}}
\newcommand{\tLaB}{\ensuremath{\tilde{\mathbf{\Lambda}}}}
\newcommand{\tLaBB}{\ensuremath{\tilde{\mathbb{\Lambda}}}}

\newcommand{\LaB}{\ensuremath{\mathbf{\Lambda}}}
\newcommand{\LaBB}{\ensuremath{\mathbb{\Lambda}}}

\newtheorem{theorem}{Theorem}[section]
\newtheorem{mydef}[theorem]{Definition}
\newtheorem{prop}[theorem]{Proposition}
\newtheorem{lemma}[theorem]{Lemma}
\newtheorem{cor}[theorem]{Corrolary}
\newtheorem{remark}[theorem]{Remark}

\numberwithin{equation}{section}
\title{Geometric Particle-in-Cell Simulations of the Vlasov--Maxwell System in Curvilinear Coordinates\thanks {\textbf{Funding}: This work has been carried out within the framework of the EUROfusion Consortium and has received funding from the Euratom research and training programme 2014-2018 and 2019-2020 under grant agreement No 633053. The views and opinions expressed herein do not necessarily reflect those of the European Commission. B.~P.~has also been
supported by Deutsche Forschungsgemeinschaft (DFG) through the TUM International Graduate School
of Science and Engineering (IGSSE).}}

\author[1,2]{Benedikt Perse}
\author[1,2]{Katharina Kormann}
\author[1,2]{Eric Sonnendr{\"u}cker}
\affil[1]{Max Planck Institute for Plasma Physics, Boltzmannstr.~2, 85748 Garching, Germany (\texttt{\{benedikt.perse,katharina.kormann,eric.sonnendruecker\}@ipp.mpg.de}).}
\affil[2]{Department of Mathematics, Technical University of Munich, Boltzmannstr.~3, 85748 Garching, Germany.}

\begin{document}
\maketitle

\begin{abstract}
Numerical schemes that preserve the structure of the kinetic equations can provide stable simulation results over a long time. An electromagnetic particle-in-cell solver for the Vlasov--Maxwell equations that preserves at the discrete level the non-canonical Hamiltonian structure of the Vlasov--Maxwell equations has been presented in [Kraus et al. 2017]. Whereas the original formulation has been obtained for Cartesian coordinates, we extend the formulation to curvilinear coordinates in this paper. For the discretisation in time, we discuss several (semi-)implicit methods either based on a Hamiltonian splitting or a discrete gradient method combined with an antisymmetric splitting of the Poisson matrix and discuss their conservation properties and computational efficiency.
\end{abstract}

\section{Introduction}

Particle-in-cell (PIC) simulations of the Vlasov--Maxwell system are an important tool to understand the evolution of a plasma in its self-consistent field. Structure-preserving PIC discretisations of the Vlasov--Maxwell system have been an active area of research in recent years (see the review \cite{Morrison2017} and references therein). The conservation properties of the Vlasov--Maxwell system can be related to the symmetry in its variational or Hamiltonian structure. Squire, Qin \& Tang \cite{Squire:2012} have derived a fully discrete geometric PIC method for the Vlasov--Maxwell system based on a discrete action principle applied to Low's Lagrangian \cite{Low:1958}. Based on this work, a systematic framework, called the geometric electromagnetic particle-in-cell method (GEMPIC), has been presented by Kraus, Kormann, Morrison \& Sonnendr\"ucker \cite{kraus2016gempic}, where the fields are discretised based on finite element exterior calculus \cite{arnold2006finite} and the equations of motions are derived from a discretisation of the Morrison--Weinstein bracket \cite{morrison1980maxwell,weinstein1981comments}. A similar formulation has been obtained simultaneously by He, Sun, Qin \&  Liu \cite{he2016}.

 So far, the work on geometric methods for the Vlasov equation has been limited to Cartesian geometry. However, in magnetic fusion, the geometry of the problem is usually given by a tokamak or stellerator device. This motivates the research reported in this paper that aims at an extension of the GEMPIC framework to curvilinear coordinates.

A number of curvilinear PIC codes has been proposed in the literature. The first electromagnetic PIC code for non-orthogonal grids was already presented by Eastwood, Arter, Brealey \& Hockney \cite{eastwood1995body} in 1995 based on a finite element description of the fields and a particle pusher in logical coordinates. This method is charge-conserving. Following up on this work, Wang, Kondrashov, Liewer \& Karmesin \cite{wang1999three} presented a 3D3V electromagnetic particle-in-cell code, called EMPIC, with a finite volume discretisation of Maxwell's equations on a deformable grid. The particles are pushed with a hybrid pusher, i.e.~the position in configuration space is updated in logical coordinates while the velocity is updated in physical coordinates. Fitchtl, Finn \& Cartwright \cite{fichtl2012arbitrary}, on the other hand, proposed an electrostatic PIC code in 2D2V with a particle pusher that uses again logical coordinates for the whole phase-space. The field solver is based on finite differences and the model preserves momentum. Delzanno et al.~use a hybrid particle pusher in their electrostatic curvilinear code CPIC \cite{delzanno2013cpic}. Their field solver is based on finite differences and allows for mesh refinement.

In a more recent work \cite{chacon2016curvilinear}, Chacon \& Chen proposed a curvilinear electromagnetic 2D3V PIC code  that conserves both charge and energy for the Vlasov--Darwin system. The algorithm is fully implicit and based on finite differences for the fields and a hybrid particle pusher. In order to speed up the non-linear iterations, the authors propose a fluid preconditioner of the system. Numerical experiments are shown for simulations on a Colella mesh with periodic boundaries.

The outline of this paper is as follows: In the next section, we review the Vlasov--Maxwell system and introduce our notation for the curvilinear coordinates. Section \ref{sec:discretisation} introduces the structure-preserving semi-discretisation based on the transformation of the de Rham complex from logical to physical coordinates. The structure of this semi-discretisation is analysed in Section \ref{sec:structure}. Various options for the discretisation of the time variable are discussed in Section \ref{sec:time} and numerical experiments that confirm the good conservation properties of our methods are shown in Section \ref{sec:numerics}, followed by some concluding remarks in Section \ref{sec:conclusions}.

\section{The Vlasov--Maxwell system and curvilinear coordinates}
The Vlasov equation in physical phase-space coordinates $(\xb,\vb)$ for a species $s$ with charge $q_s$ and mass $m_s$ reads
\begin{align}\label{V}
&\frac{\partial f_s(\xb,\vb,t)}{\partial t}+ \vb \cdot \nabla_{\xb} f_s(\xb,\vb,t)+\frac{q_s}{m_s} (\Eb(\xb,t)+\vb\times \Bb(\xb,t))\cdot \nabla_\vb f_s(\xb,\vb,t)=0,
\end{align}
where $\Eb$ and $\Bb$ denote the electromagnetic fields, which are evolved according to Maxwell's equations.
The system couples through the first two moments of the particle distribution function $f_s$,  the charge and current density,
\begin{align*}
\rhob(\xb,t)=\sum_s q_s \int f_s(\xb,\vb,t) \du\vb, \ \mathbf{J}(\xb,t)=\sum_s q_s \int f_s(\xb,\vb,t) \vb \du\vb.
\end{align*}

The equations of motion can be obtained by a bilinear, antisymmetric Poisson bracket that satisfies Leibniz' rule and the Jacobi identity and was introduced in \cite{morrison1980maxwell} and corrected in \cite{weinstein1981comments}. It is defined as
\begin{align*}
\{\mathcal{F},\mathcal{G}\}[f_s,\mathbf{E},\mathbf{B}]&= \sum_s \int \left[\frac{\delta \mathcal{F}}{\delta f_s},\frac{\delta \mathcal{G}}{\delta f_s}\right] \du\xb\vb\\
&+\sum_s \frac{q_s}{m_s} \int f_s\left(\nabla_{\vb} \frac{\delta \mathcal{F}}{\delta f_s} \cdot \frac{\delta \mathcal{G}}{\delta \Eb}-\nabla_{\vb} \frac{\delta \mathcal{G}}{\delta f_s} \cdot \frac{\delta \mathcal{F}}{\delta \mathbf{E}}\right)\du\xb\vb\\
&+ \sum_s \frac{q_s}{m_s^2} \int f_s \mathbf{B} \cdot \left( \nabla_{\vb} \frac{\delta \mathcal{F}}{\delta f_s} \times \nabla_{\vb} \frac{\delta \mathcal{G}}{\delta f_s} \right)\du\xb\vb\\
&+ \int \left(\operatorname{curl} \frac{\delta \mathcal{F}}{\delta\mathbf{E}} \cdot \frac{\delta \mathcal{G}}{\delta\mathbf{B}}-\operatorname{curl} \frac{\delta \mathcal{G}}{\delta\mathbf{E}} \cdot \frac{\delta \mathcal{F}}{\delta\mathbf{B}} \right) \du\xb,
\end{align*}
where $[f,g]=\nabla_\xb f \cdot \nabla_\vb g-\nabla_\xb g \cdot \nabla_\vb f.$

The time evolution of a functional  $\mathcal{F}[f_s,\mathbf{E,B}]$ is given by
\begin{align} \label{timeevolution}
\frac{\du}{\du t}\mathcal{F}[f_s,\mathbf{E,B}]=\{\mathcal{F},\mathcal{H}\},
\end{align} 
where the Hamiltonian $\mathcal{H}$ is given by the sum of the kinetic energy of the particles and the electric and magnetic field energies,

\begin{align} \label{Hamiltonian}
\mathcal{H}=\sum_s \frac{m_s}{2} \int |\vb|^2 f_s(\mathbf{x,v}) \ \du\xb \du\vb+\frac{1}{2}\int \left( |\mathbf{E}(\xb)|^2+|\mathbf{B}(\xb)|^2\right) \du\xb.
\end{align}

\subsection{Differential forms and the structure of Maxwell's equations}
The structure of Maxwell's equations can be understood by interpreting the fields as differential forms following \cite{bossavitdifferential,bachman2012geometric, hiptmair2002finite}.
Therefore, we first reformulate Maxwell's equations in terms of the electric field ($\Eb$), the magnetic field ($\Hb$), the displacement field ($\Db$), and the magnetic induction ($\Bb$),
\begin{align*} 
&\frac{\partial \Db(\xb,t)}{\partial t} = \nabla_{\xb} \times \Hb(\xb,t) - \Jb(\xb,t), 
&&\frac{\partial  \Bb(\xb,t)}{\partial t}=-\nabla_{\xb} \times \Eb(\xb,t),
\\ 
&\nabla_{\xb} \cdot \Db(\xb,t)=\rho(\xb,t),
&&\nabla_{\xb} \cdot  \Bb(\xb,t)=0.
\end{align*} 
The spaces of electromagnetics form a de Rham complex, which in terms of Sobolev spaces can be expressed as
\begin{equation}\label{CGL_g}
\begin{tikzpicture}[baseline=(current  bounding  box.center)]
  \matrix (m) [matrix of math nodes,row sep=3em,column sep=4em,minimum width=2em] {
     H^1(\Omega) & H(\curl,\Omega) & H(\div,\Omega) & L^2(\Omega)
    \\
    };
  \path[-stealth]
    (m-1-1) edge node [above] {$\grad$} (m-1-2)
    (m-1-2) edge node [above] {$\curl$} (m-1-3)
    (m-1-3) edge node [above] {$\div$}  (m-1-4)
     ;
\end{tikzpicture}
\end{equation}
and is accompanied by the dual de Rham complex
\begin{equation}\label{CGL_g_dual}
\begin{tikzpicture}[baseline=(current  bounding  box.center)]
  \matrix (m) [matrix of math nodes,row sep=3em,column sep=4em,minimum width=2em] {
     {L^2}^\star(\Omega) & {H}^\star(\div,\Omega) & {H}^\star(\curl,\Omega) & {H^1}^\star(\Omega)
    \\
    };
  \path[-stealth]
    (m-1-1) edge node [above] {$\grad$} (m-1-2)
    (m-1-2) edge node [above] {$\curl$} (m-1-3)
    (m-1-3) edge node [above] {$\div$}  (m-1-4)
     ;
\end{tikzpicture}
\end{equation}


The complex property is satisfied, since $ \curl \grad \Phi=0$ for all $\Phi \in H^1(\Omega)$ and $ \div  \curl \Ab =0$ for all $\Ab \in  H(\curl,\Omega)$.

In the interpretation of the fields, there are two options, one yielding a strong formulation of Amp\`ere's law and the other yielding a weak formulation. We choose the latter and consider $\Bb\in H(\div,\Omega)$ as primal 2-form, $\Eb\in H(\curl,\Omega)$ as primal 1-form, $\Hb \in {H}^\star(\curl,\Omega) $ as dual 1-form, $\Db, \Jb \in {H}^\star(\div,\Omega) $ as dual 2-form and $\rho \in {L^2}^\star(\Omega)$ as dual 3-form. 
Hence, Amp\`ere's law and the electric Gauss law are interpreted in the weak sense with $\Eb$ representing $\Db$ and $\Bb$ representing $\Hb$. From now on, we will consider Maxwell's equations in the following mixed form. Here, we obtain for all $\boldsymbol{\varphi} \in H(\curl,\Omega), \psi \in H^1(\Omega)$

\begin{align*}
\int_{\Omega}  \boldsymbol{\varphi}(\xb,t) \cdot \frac{\partial \Eb(\xb,t)}{\partial t} \du \xb &= \int_{\Omega}  \nabla_{\xb} \times \boldsymbol{\varphi}(\xb,t)  \cdot \Bb(\xb,t) \du \xb -   \Jb(\boldsymbol{\varphi})(t)
\\ 
\frac{\partial  \Bb(\xb,t)}{\partial t}&=-\nabla_{\xb} \times \Eb(\xb,t),
\\
-\int_{\Omega} \nabla_{\xb} \psi(\xb,t) \cdot   \Eb(\xb,t)&= \rho( \psi)(t),
\\ 
\nabla_{\xb} \cdot  \Bb(\xb,t) &=0,
\end{align*} 
where $\Jb$ and $\rho$ are  linear functionals. Note that we have assumed that the boundary terms vanish.

\subsection{Curvilinear coordinates}
\subsubsection{Notation}
Let us first define the notation for the curvilinear coordinates before discussing how these can be consistently combined with differential forms. We consider a bijective coordinate transformation from the logical space $[0,1]^3$ to the physical space $\Omega$. The transformation map is denoted by 
\begin{align*}
F \colon [0,1]^3=:\tom \to \Omega \subset \R^3, \
F(\xib)=\xb
\end{align*}
and $\xib=(\xi_1,\, \xi_2,\,\xi_3)^\top$ and $\xb=(
x_1,\, x_2,\, x_3)^\top$ are the variables on the logical and physical mesh, respectively. 
The matrix of the partial derivatives, the Jacobi matrix, and its determinant, the Jacobian, are defined as
\begin{align*}
\left(DF(\xib)\right)_{ij}&=\frac{\partial F_i(\xib)}{\partial \xi_j}=\frac{\partial x_i}{\partial \xi_j}, \quad
J_F(\xib)=\det(DF(\xib)).
\end{align*}
We assume that the mapping is non-singular, i.e.~$J_F(\xib)\neq 0$ for all $\xib\in \tom$, and therefore, the Jacobi matrix is invertible.
\begin{mydef}
The column vectors of the Jacobi matrix form the so called covariant basis of the tangent space 
\begin{align*}
\bt_i=\frac{\partial F(\xib)}{\partial \xi_i}=\frac{\partial \xb}{\partial \xi_i}, \quad DF=(\bt_1|\bt_2|\bt_3)
\end{align*}
whereas the columns of the transposed inverse Jacobi matrix form the dual basis, which is called the contravariant basis of the cotangent space,
\begin{align*}
DF(\xib)^{-\top}=: N(\xib)=(\bn_1|\bn_2|\bn_3).
\end{align*} 
\end{mydef}

\begin{mydef}
The coefficients of the metric $G_t$ and the inverse metric $G_n$ are defined in the following symmetric way
\begin{align*}
G_t(\xib)=DF(\xib)^\top DF(\xib), \quad G_n(\xib) = N(\xib)^\top N(\xib).
\end{align*}
\end{mydef}

\subsubsection{Transformation of differential forms}
We introduce curvilinear coordinates to the differential forms and show how they are transformed in a consistent way as can also be seen in \cite{kreeft2011mimetic}.

\begin{mydef}
For a scalar differential 0-form, $ g\in H^1(\Omega)$, we define $\tilde{g}\in H^1(\tom)$ as
\begin{align}\label{0form}
\tilde g(\xib):= g(F(\xib))=g(\xb).
\end{align}
\end{mydef}
Next, we consider the transformation of the other differential forms.
\begin{prop}\label{derivtrafo} 
\begin{enumerate}
\item A vector function, $\Eb\in H(\curl,\Omega)$, corresponding to a differential 1-form, is transformed by the covariant Piola transform
\begin{align}\label{PiolaE}
\Eb(\xb)=N(\xib) \tEb(\xib)\text{ with } \tEb \in H(\curl,\tom).
\end{align}
\item A vector function,  $\Bb\in H(\div,\Omega)$, corresponding to a differential 2-form, is transformed by the contravariant Piola transform
\begin{align}\label{PiolaB}
\Bb(\xb)=\frac{DF(\xib)}{J_F(\xib)}\tBb(\xib) \text{ with } \tBb \in H(\div,\tom).
\end{align}
\item A scalar differential 3-form, $h \in L^2(\Omega)$, is related to $\tilde h \in L^2(\tom)$ via
\begin{align*}
h(\xb)=\frac{1}{J_F(\xib)}\tilde{h}(\xib).
\end{align*}  
\end{enumerate}
\end{prop}
\begin{proof}
 The proof can be found in \cite[sec. 1.7]{AdvancedFEM} by considering the representation of the fields through the electromagnetic potentials. 
\end{proof}

\subsection{Curvilinear Vlasov and Maxwell's equations}

We use the curvilinear transformation rules for the differential forms in order to transform the Vlasov--Maxwell system. 

\begin{prop}
Under the coordinate transformation $F(\xib)=\xb$, 
\begin{enumerate}
\item the Vlasov equation \eqref{V} transforms to
\begin{align*}
\frac{\partial \tf_s(\xib,\vb,t)}{\partial t}+& N(\xib)^\top \vb  \cdot \nabla_{\xib} \tf_s(\xib,\vb,t)
\\+&\frac{q_s}{m_s} N(\xib) \left(\tEb(\xib,t)+( N(\xib)^\top\vb)\times \tBb(\xib,t)\right)\cdot \nabla_\vb \tf_s(\xib,\vb,t)=0;
\end{align*}
\item Faraday's and magnetic Gauss' law in strong form do not change, i.e.~
\begin{subequations}
\begin{align}\label{strongFaraday}
 \frac{\partial \tBb(\xib,t)}{\partial t}&=-  \nabla_{\xib} \times \tEb(\xib,t), 
 \\ \label{strongmGauss}
 \nabla_\xib \cdot \tBb(\xib,t) &=0; 
\end{align}
\end{subequations}
\item the weak formulation of Amp\`ere's law and Gauss' law becomes for all $\tilde{\boldsymbol{\varphi}}\in H(\operatorname{curl},\tom)$, $\tilde \psi \in H^1(\tom)$
\begin{subequations}
\begin{align}\label{weakAmpere}
&\frac{\du}{\du t}\int_{\tom}  N \tilde{\boldsymbol{\varphi}} \cdot N \tEb |J_F| \du\xib = \int_{\tom}   \frac{DF}{J_F} \nabla_\xib \times \tilde{\boldsymbol{\varphi}} \cdot \frac{DF}{J_F} \tBb |J_F| \du\xib -\tJb( N \tilde{\boldsymbol{\varphi}}), 
\\\label{weakGauss}
&-\int_{\tom} N \nabla_\xib \tilde \psi \cdot N \tEb |J_F| \du\xib =  \tilde \rho(\tilde \psi).   
\end{align}
\end{subequations}
\end{enumerate}
\end{prop}
\begin{proof}
The equations can be derived by inserting the coordinate transformation into the Vlasov and Maxwell equations and using Prop.~\ref{derivtrafo}.
\end{proof}

\section{Structure-preserving discretisation in curvilinear geometry}\label{sec:discretisation}

In this section, we will introduce a particle discretisation for the distribution function and a compatible finite element discretisation of the fields, extending the discretisation proposed in \cite{kraus2016gempic} to the curvilinear case.

\subsection{Discrete particle distribution function}
In order to define the charge and current densities in Maxwell's equations, we need to define a discrete particle distribution function from the positions $\xb_p$ and velocities $\vb_p$ of the  $N_p$ particles of all species $s$. We use a definition based on $\delta$-functions in phase space,
\begin{align*}
f_h(\xb,\vb,t)=\sum_{p=1}^{N_p} \omega_p \delta(\xb-\xb_p(t))  \delta(\vb-\vb_p(t)).
\end{align*}
The $\delta$-function defines the point evaluation in a convolution with another function. Therefore, we need to scale by the inverse Jacobian when the argument is transformed. This yields the following definition of the distribution function in logical coordinates,
\begin{equation}
\begin{aligned}\label{pic}
\tf_h(\xib,\vb,t)=f_h(F(\xib),\vb,t)&=\sum_{p=1}^{N_p} \omega_p \delta(F(\xib)-F(\xib_p(t))) \delta(\vb-\vb_p(t))\\
&=\sum_{p=1}^{N_p} \omega_p \frac{\delta(\xib-\xib_p(t))}{|J_F(\xib_p)|} \delta(\vb-\vb_p(t)).
\end{aligned}
\end{equation}
Upon inserting this discrete form of the particle distribution function, the current and the charge take the following form,
\begin{subequations}
\begin{align}\label{current}
\tilde{\mathbf{J}}_h(\xib)&= \int q_s \tf_h(\xib,\vb,t)\vb \ \du\vb=  \sum_{p=1}^{N_p} q_p \omega_p \frac{\delta(\xib-\xib_p)}{|J_F(\xib_p)|}\vb_p,
\\\label{charge}
\trhob_h(\xib)&= \int q_s \tf_h(\xib,\vb,t) \ \du\vb= \sum_{p=1}^{N_p}q_p \omega_p \frac{\delta(\xib-\xib_p)}{|J_F(\xib_p)|}.
\end{align}
\end{subequations}
Note that this representation is smooth enough, since we only consider the densities in weak form.

Let us collect the logical positions of all particles and their velocities in the vectors $ \Xib:=(\xib_1,...,\xib_{N_P})^\top, \Vb:=(\vb_1,...,\vb_{N_P})^\top.$ Moreover, we use use the following notation, $\MM_m:=\operatorname{diag}(\omega_p m_p)\otimes \mathbb{I}_3, \MM_q:=\operatorname{diag}(\omega_p q_p)\otimes \mathbb{I}_3, \NN:=\operatorname{diag}(N(\xib_p)),1\leq p\leq N_p.$ 
Hence, the equations for the characteristics of the particles can be written as
\begin{equation}
\begin{aligned}\label{discreteVlasov}
\dot{\Xib}&=\NN(\Xib)^\top\Vb,
\\
\dot{\Vb}&= \MM_q \MM_m^{-1} \NN(\Xib) \left( \tEb(\Xib,t)+ (\NN(\Xib)^\top \Vb) \times \tBb(\Xib,t)\right).
\end{aligned}
\end{equation}

\subsection{Finite element discretisation}
\subsubsection{Discrete de Rham sequence}
Arnold, Falk \& Winter \cite{arnold2006finite} have developed a theoretical framework for the finite element discretisation that respects the sequence properties of the de Rham complex. The idea is to define discrete spaces that form the following commuting diagram with the continuous spaces

\begin{equation*}
\begin{tikzpicture}[baseline=(current  bounding  box.center)]
  \matrix (m) [matrix of math nodes,row sep=3em,column sep=4em,minimum width=2em] {
           H^1(\tom) & H(\curl,\tom) & H(\div,\tom) & L^2(\tom)
    \\
           \tilde{V}_0 & \tilde{V}_1 & \tilde{V}_2 & \tilde{V}_3
    \\
    };
  \path[-stealth]
    (m-1-1) edge node [left]  {$\Pi_0$} (m-2-1)
    (m-1-1) edge node [above] {$\grad$} (m-1-2)
    (m-1-2) edge node [left]  {$\Pi_1$} (m-2-2)
    (m-2-1) edge node [above] {$\grad$} (m-2-2)
    (m-1-3) edge node [left]  {$\Pi_2$} (m-2-3)
    (m-1-4) edge node [left]  {$\Pi_3$} (m-2-4)
    (m-1-2) edge node [above] {$\curl$} (m-1-3)
    (m-2-2) edge node [above] {$\curl$} (m-2-3)
    (m-1-3) edge node [above] {$\div$}  (m-1-4)
    (m-2-3) edge node [above] {$\div$}  (m-2-4)
    ;
\end{tikzpicture}
\end{equation*}
The  operators $\Pi_k$, $k=0,1,2,3$, are projecting the corresponding differential forms to the finite dimensional subspaces $\tilde{V}_k$ with dimension, 
\begin{align*}\dim \tilde{V}_k=\left\{\begin{matrix} N_k &  \text{ if } k = 0,3 \\
3 N_k &  \text{ if } k = 1,2
\end{matrix}\right., \text{ i.e.~}
\Pi_0 \tPhi=\tPhi_h \in \tilde{V}_0, \Pi_1\tEb=\tEb_h\in \tilde{V}_1, \Pi_2\tBb=\tBb_h\in \tilde{V}_2.
\end{align*}
We introduce basis functions for the finite dimensional subspaces $\tilde{V}_k$, scalar functions $\tLa^k_i$ for $k=0,3$ and vector valued functions
\begin{align*}
\tLab^k_{i,1}=(
\tLa^{k,1}_i,0,0)^\top, 
\tLab^k_{i,2}= (0,\tLa^{k,2}_i,0)^\top, 
\tLab^k_{i,3}=(0,0, \tLa^{k,3}_i)^\top \text{ for } k=1, 2.
\end{align*}
The de Rham structure can also be expressed on the level of matrices and vectors. For some $\xib \in \tom$, we collect the value of each basis function in a row vector as 
\begin{align*}
\tLam^k(\xib) &=\left( \tLa^k_{1}(\xib) ,\tLa^k_{2}(\xib) ,..., \tLa^k_{N_k}(\xib)  \right) \in \R^{1\times N_k} \text{ for }  k=0,3, \\
\tLaB^k(\xib)&=\left( \tLab^k_{1,1}(\xib),\tLab^k_{1,2}(\xib),..., \tLab^k_{N_k,3}(\xib) \right) \in \R^{3\times 3N_k} \text{ for }  k=1,2.
\end{align*}
Then, the following relations hold for the basis functions,
\begin{equation}\label{diffop}
\begin{aligned}
\nabla_{\xib}\tLam^0(\xib) =  \tLaB^1(\xib) \G, 
\quad
\nabla_{\xib} \times \tLaB^1(\xib) =  \tLaB^2(\xib) \C,
\quad
\nabla_{\xib} \cdot \tLaB^2(\xib) =  \tLam^3(\xib) \D,
\end{aligned}
\end{equation}
for some matrix $\G \in \R^{3N_1\times N_0}$ denoting the discrete gradient matrix, $\C \in \R^{3N_2\times 3N_1}$ denoting the discrete curl matrix and $\D \in \R^{N_3\times 3N_2}$ denoting the discrete divergence matrix, all independent of $\xib$. These matrices need to satisfy 
\begin{align} \label{discomp}
\D \C = 0 \text{ and } \C \G =0 
\end{align}
to mimic the complex properties $\div \curl = 0$ and $\curl \grad = 0$.

%

\begin{lemma}
From the de Rham sequence on the logical mesh, a de Rham sequence can be constructed on the physical domain by
\begin{align*}
\Lambda^0(\xb) = \tLam^0(\xib), \ \LaB^1(\xb) = N(\xib)\tLaB^1(\xib), \ \LaB^2(\xb) = \frac{DF(\xib)}{J_F(\xib)}\tLaB^2(\xib), \ \Lambda^3(\xb) = \frac{1}{J_F(\xib)}\tLam^3(\xib).
\end{align*}
\end{lemma}
\begin{proof}
The following omputations show the assertions:
\begin{align*}
\nabla_{\xb} \Lambda^0(\xb) &= N(\xib) \nabla_{\xib} \tLam^0(\xib) = N(\xib) \tLaB^1(\xib) \G = \LaB^1(\xb) \G,
\\
\nabla_{\xb} \times \LaB^1(\xb) &= \frac{DF(\xib)}{J_F(\xib)} \nabla_{\xib} \times \tLaB^1(\xib)  = \frac{DF(\xib)}{J_F(\xib)} \tLaB^2(\xib) \C = \LaB^2(\xb) \C,
\\
\nabla_{\xb} \cdot \LaB^2(\xb) &= \frac{1}{J_F(\xib)} \nabla_{\xib} \cdot \tLaB^2(\xib)  = \frac{1}{J_F(\xib)} \tLam^3(\xib) \D = \Lambda^3(\xb) \D,
\end{align*}
for the same matrices $ \G$, $\C$ and $\D$ as on the logical mesh.
\end{proof}	
The mass matrices for differential forms are defined as
\begin{equation}
\begin{aligned} \label{mass}
(\tM_0)_{ij}&= \int_{{\tom}} \tLa^0_i \tLa^0_j |J_F| \ \du\xib \text{ for } 1 \leq i,j \leq N_0, 
  \\
 (\tM_1)_{IJ}&= \int_{{\tom}} (\tLab^1_I)^\top G_n \tLab^1_{J} |J_F| \ \du\xib \text{ for } 1 \leq I,J \leq 3 N_1, 
  \\
  (\tM_2)_{IJ}& = \int_{{\tom}} (\tLab^2_I)^\top G_t \tLab^2_J \frac{1}{|J_F|} \ \du\xib
 \text{ for } 1 \leq I,J \leq 3N_2,
 \\
(\tM_3)_{ij}&= \int_{{\tom}} \tLa^3_i  \tLa^3_j  \frac{1}{|J_F|} \ \du\xib \text{ for } 1 \leq i,j \leq N_3. 
\end{aligned}
\end{equation}

\subsubsection{Discretisation of the curvilinear Maxwell equations}
To discretise Maxwell's equations based on the compatible finite element spaces, we represent the electromagnetic fields with the $3N_k$, $k=1,2$, degrees of freedom as 
\begin{subequations}
\begin{align}\label{ef}
&\tEb_h(\xib,t)=\tLaB^1(\xib) \teb(t)= \sum_{j=1}^{N_1} \sum_{i=1}^3 \tLab^1_{j,i}(\xib) \te_{j,i}(t), 
\\\label{bf} 
&\tBb_h(\xib,t)=\tLaB^2(\xib) \tbb(t)= \sum_{k=1}^{N_2} \sum_{i=1}^3 \tLab^2_{k,i}(\xib) \tb_{k,i}(t). 
\end{align}
\end{subequations}
We recapitulate the Piola transform \eqref{PiolaE}, \eqref{PiolaB} of the electromagnetic fields and introduce the basis representation of the finite-dimensional subspaces
\begin{align*} 
\Eb_h(\xb,t)&=\Eb_h(F(\xib),t) = N(\xib) \tEb_h(\xib,t)=  N(\xib) \tLaB^1(\xib) \teb(t),
\\
\Bb_h(\xb,t)&=\Bb_h(F(\xib),t)= \frac{DF(\xib)}{J_F(\xib)} \tBb_h(\xib,t)=\frac{DF(\xib)}{J_F(\xib)} \tLaB^2(\xib) \tbb(t).
\end{align*}

\begin{prop}The transformed discrete version of the Amp\`ere and the electric Gauss law take the following form in matrix notation
\begin{subequations}
\begin{align}\label{discreteAmpere}
\tM_1 \dot{\teb}&= \C^\top \tM_2 \tbb-  \MM_q  \tLaBB^1(\Xib)^\top \NN(\Xib)^\top\Vb,
\\ \label{discreteGauss}
\G^\top\tM_1 \teb&= - \MM_q \tLaBB^0(\Xib)^\top \mathbbm{1}_{N_p}.
\end{align}
\end{subequations}
\end{prop}
\begin{proof}
The Maxwell equations in weak form are discretised by approximating $(\tEb,\tBb)\in H(\curl,\tom)\times H(\div,\tom)$ with the discrete fields $(\tEb_h,\tBb_h)\in \tilde{V}_1 \times \tilde{V}_2$ and discrete test functions in $\tilde V_0$ and $\tilde V_1$.

For Amp\`ere's law, we insert \eqref{ef} and \eqref{bf} into \eqref{weakAmpere} and use the basis functions $\tLaB^1 \in \tilde V_1$ as test functions
\begin{align*} 
&\frac{\du}{\du t} \int_{\tom} \left(N(\xib) \tLaB^1(\xib)\right)^\top N(\xib)\tLaB^1(\xib) \teb  |J_F(\xib)| \du\xib\\
=&\int_{\tom} \left(\frac{DF(\xib)}{J_F(\xib)} \nabla_{\xib} \times \tLaB^1(\xib)\right)^\top \frac{DF(\xib)}{J_F(\xib)}\tLaB^2(\xib) \tbb |J_F(\xib)|  \du\xib- \tJb( N \tLaB^1) .
\end{align*}
Next, we use the relation \eqref{diffop} for the $\curl$  and insert the transformed current \eqref{current}
\begin{align*}
&\int_{\tom} \tLaB^1(\xib)^\top G_n(\xib) \tLaB^1(\xib) |J_F(\xib)|  \du\xib \ \dot{ \teb }  \\=&\C^\top \int_{\tom}  \tLaB^2(\xib)^\top G_t(\xib) \tLaB^2(\xib) \frac{1}{|J_F(\xib)|}  \du\xib \  \tbb
- \sum_{p=1}^{N_p} q_p\omega_p  \tLaB^1(\xib_p)^\top N(\xib_p)^\top  \vb_p. 
\end{align*}
For Gauss' law, we insert \eqref{ef} into \eqref{weakGauss} and choose the basis functions $\tLam^0\in \tilde{V}_0$ as test functions
\begin{align*}
-\int_{\tom} \left( N(\xib) \nabla_{\xib} \tLam^0(\xib) \right)^\top N(\xib)\tLaB_1(\xib) \teb  |J_F(\xib)| \ \du\xib
=  \tilde \rho((\tLam^0)^\top).
\end{align*}
Then, we use the relation \eqref{diffop} for the gradient and insert the transformed charge \eqref{charge}
\begin{align*}\G^\top \int_{\tom} \tLaB^1(\xib)^\top G_n(\xib) \tLaB^1(\xib) |J_F(\xib)|  \du\xib \ \teb & = 
- \sum_{p=1}^{N_p} q_p \omega_p \tLam^0(\xib_p)^\top.
\end{align*}
With the notation from \eqref{mass} we obtain the equations in matrix notation. 
\end{proof}

\begin{prop} The transformed discrete version of Faraday's and magnetic Gauss' law take the following form in matrix notation
\begin{subequations}
\begin{align}\label{discreteFaraday}
\dot {\tbb} &= -  \C \teb,
\\ \label{discretemGauss}
\D \tbb&=0.
\end{align}
\end{subequations}
\end{prop}
\begin{proof}
We insert the discrete transformed fields and their basis representation \eqref{ef}, \eqref{bf} into \eqref{strongFaraday} and \eqref{strongmGauss}. Then, \eqref{discreteFaraday} and \eqref{discretemGauss} follow, since $\tLaB^2$ and $\tLaB^3$ are a basis, respectively.
\end{proof}

\section{Semi-discrete Hamiltonian structure}\label{sec:structure}
In the previous section, we have obtained a spatial semi-discretisation of the Vlasov--Maxwell system. Let us now analyse the structure of this semi-discretisation.

\subsection{Equations of motion and Poisson matrix}
From the discretisation of the Vlasov--Maxwell system \eqref{discreteVlasov}, \eqref{discreteAmpere} and \eqref{discreteFaraday} we get the following equations of motion with hybrid particle push,
\begin{equation}
\begin{aligned} \label{equations of motion}
\dot{\Xib}&=\NN(\Xib)^\top\Vb, \\
\dot{\Vb}&=\MM_m^{-1}\MM_q \NN(\Xib) \left( \tLaBB^1(\Xib) \teb+ (\NN(\Xib)^\top \Vb) \times \tLaBB^2(\Xib) \tbb \right),\\
\tM_1 \dot{\teb} &=\C^\top\tM_2 \tbb -  \tLaBB^1(\Xib)^\top \NN(\Xib)^\top \MM_q  \Vb,\\
\dot{\tbb}&=-\C \teb,
\end{aligned}
\end{equation}
where we denote by $\tLaBB^1(\Xib)$ the $3N_p\times 3N_1$ matrix with generic term $\tLaB^1_I(\xib_p)$ for $1\leq p\leq N_p$ and  $ 1\leq I \leq 3N_1$. Furthermore, we introduce the $N_p\times N_0$ matrix $\tLaBB^0(\Xib)$ with generic term $\tLam^0_i(\xib_p)$ for $1\leq p\leq N_p, 1\leq i \leq N_0$. $\tLaBB^2$ and $\tLaBB^3$ are defined accordingly.
 
The corresponding divergence constraints were discretised in \eqref{discreteGauss} and \eqref{discretemGauss} as
\begin{equation}\label{divconstraints}
\begin{aligned}
\G^\top \tM_1 \teb &= -\MM_q \tLaBB^0(\Xib)^\top \mathbbm{1}_{N_p},
\\
\D \tbb &= 0.
\end{aligned}
\end{equation}

\begin{remark}
Note that it would also be possible to transform the velocity into logical coordinates,  $\tilde{\Vb}=\NN(\Xib)^\top\Vb$. However, this leads to the velocity push $ \dot{\tilde{\Vb}}= \NN^\top \dot{\Vb} + \frac{\partial (\NN(\Xib(t))^\top \Vb)}{\partial t}$, which contains a derivative tensor.
\end{remark}
Let us consider the semi-discrete Hamiltonian that corresponds to system \eqref{equations of motion}.
\begin{prop}
The semi-discrete Hamiltonian can be written in matrix notation as
\begin{align}\label{disHam}
\tilde{\mathcal{H}}_h= \frac{1}{2}\Vb^\top \MM_m \Vb+\frac{1}{2}\teb^\top \tM_1 \teb+\frac{1}{2}\tbb^\top \tM_2 \tbb.
\end{align}
\end{prop}
\begin{proof}
We transform the Hamiltonian \eqref{Hamiltonian} to curvlinear coordinates,
\begin{align*}
\tilde{\mathcal{H}}=& \sum_s \frac{m_s}{2} \int |\vb|^2 \tf_s (\xib,\vb,t) |J_F(\xib)| \ \du\xib \du\vb\\+&\frac{1}{2}\int \left( | N(\xib) \tEb(\xib)|^2+|\frac{DF(\xib)}{J_F(\xib)}\tBb(\xib)|^2\right) |J_F(\xib)| \du\xib.
\end{align*}
For the semi-discrete version, we introduce the ansatz for the discrete particle distribution function \eqref{pic} and the basis representation of the discrete fields \eqref{ef},\eqref{bf} to find
\begin{align*}
\tilde{\mathcal{H}}_h= \sum_{p=1}^{N_p} \frac{m_p}{2} \vb_p^2 +&\frac{1}{2}\int  N(\xib) \tLaB^1(\xib) \teb \cdot N(\xib) \tLaB^1(\xib) \teb |J_F(\xib)|  \du\xib \\
+ &\frac{1}{2} \int \frac{DF(\xib)}{J_F(\xib)} \tLaB^2(\xib) \tbb \cdot \frac{DF(\xib)}{J_F(\xib)} \tLaB^2(\xib) \tbb  |J_F(\xib)| \du\xib.
\end{align*}
The assertion follows when we use the definition of the mass matrices \eqref{mass}.
\end{proof}
Then, the derivative of the discrete Hamiltonian is computed as
\begin{align*}
D\tilde{\mathcal{H}}_h(\tub)=(0,(\MM_m \Vb)^\top , (\tM_1 \teb)^\top,(\tM_2 \tbb)^\top)^\top.
\end{align*}
Next, we consider the discretisation of the Poisson bracket, which can be expressed as
\begin{equation}\label{eq:discrete_bracket}
\{\mathcal{F}_h(\tub),\mathcal{G}_h(\tub)\}_d = D \mathcal{F}_h(\tub)^\top \mathbb{J}(\tub) D\mathcal{G}_h(\tub),
\end{equation}
where $\mathbb{J}$ is the discrete Poisson matrix. In particular, setting $\mathcal{F}_h(\tub) = \tub$ and $\mathcal{G}_h(\tub) = \tilde{\mathcal{H}}_h$, the time evolution of the equations of motion is given by the discrete analogon of \eqref{timeevolution},
\begin{equation}\label{eq:equations_of_motion_poisson}
\frac{\du \tub}{\du t} = \mathbb{J}(\tub) D\tilde{\mathcal{H}}_h.
\end{equation}
For our semi-discretisation \eqref{equations of motion}, the Poisson matrix takes the form
\begin{align}\label{Poissonmatrix}
\mathbb{J}=
 \begin{pmatrix}
0& \N^\top \MM_m^{-1} &0&0\\
-\MM_m^{-1} \NN & \MM_m^{-1}\MM_q \NN \tBB \NN^\top \MM_m^{-1} & \MM_m^{-1}\MM_q \NN\tLaBB^1 \tM_1^{-1}&0\\
0&-\tM_1^{-1} (\tLaBB^1)^\top \NN^\top \MM_q \MM_m^{-1}  &0&\tM_1^{-1}\C^\top \\
0&0&-\C\tM_1^{-1}&0
\end{pmatrix},
\end{align}
where $\tBB(\mathbf{\Xib,b})$ is a $3 N_p \times 3 N_p$ block diagonal matrix with generic block 
\begin{align}\label{block_of_IxB}
\hat{\tilde B}_h(\xib_p,t)=\sum_{i=1}^{N_2} \begin{pmatrix}
0&\tb_{i,3}(t)\tLa_{i}^{2,3}(\xib_p)&-\tb_{i,2}(t)\tLa_{i}^{2,2}(\xib_p)\\
-\tb_{i,3}(t)\tLa_{i}^{2,3}(\xib_p)&0&\tb_{i,1}(t)\tLa_{i}^{2,1}(\xib_p)\\ \tb_{i,2}(t)\tLa_{i}^{2,2}(\xib_p)&-\tb_{i,1}(t)\tLa_{i}^{2,1}(\xib_p)&0
\end{pmatrix}.
 \end{align}

\subsection{Discrete Poisson bracket}
In this section, we show that, with this form of the Poisson matrix \eqref{Poissonmatrix},  \eqref{eq:discrete_bracket} indeed defines a (discrete) Poisson bracket. 
\begin{theorem}
The differential operator $\{f,g\}_d= Df^\top \mathbb{J} Dg$ forms a discrete Poisson bracket.
\end{theorem}
\begin{proof}
The Poisson matrix $\mathbb{J}$ is obviously antisymmetric and the bilinearity and Leibniz's rule follow trivially from the form \eqref{eq:discrete_bracket}. So, it is only left to prove that the Poisson matrix follows the Jacobi identity.

The matrix $\mathbb{J}$ satisfies the Jacobi identity if and only if the following condition holds
\begin{align*}
\sum_l \left( \frac{\partial \mathbb{J}_{ij}(\ub)}{\partial u_l} \mathbb{J}_{lk}(\ub)+ \frac{\partial \mathbb{J}_{jk}(\ub)}{\partial u_l} \mathbb{J}_{li}(\ub)+\frac{\partial \mathbb{J}_{ki}(\ub)}{\partial u_l} \mathbb{J}_{ll}(\ub) \right) = 0 \quad \forall i,j,k,
\end{align*}	
where $i,j,k,l$ run from $1$ to $6N_p+3N_1+3N_2$. The Poisson matrix $\mathbb{J}$ has the following block-structure
\begin{align*}
\mathbb{J}=\begin{pmatrix}
0&J_{12}(\Xib)&0&0\\
J_{21}(\Xib)&J_{22}(\Xib,\tbb)&J_{23}(\Xib)&0\\
0&J_{32}(\Xib)&0&J_{34}\\
0&0&J_{43}&0
\end{pmatrix}.
\end{align*}
Hence, many combinations of indices are trivially zero. In particular, the matrix only depends on $\Xib$ and $\tbb$ and hence the derivatives are only non-zero if $l \in [1,3N_p]$ or $l \in [6N_p+3N_1+1,6N_p+3N_1+3N_2]$. Moreover, we need to find combinations of $i,j,k$ (or permutations of these) for which both $\mathcal{J}_{ij}$ and $\mathcal{J}_{lk}$ are non-vanishing. For $l \in [1,3N_p]$, this only leaves the options $i,j,k \in [3N_p+1,6N_p]$ or $i \in [1,3N_p]$ and  $j,k \in [3N_p+1,6N_p]$. If $l \in [6N_p+3N_1+1,6N_p+3N_1+3N_2]$, we only have $i,j\in [3N_p+1,6N_p]$ and $k \in [6N_p+1,6N_p+3N_1]$. Let us now consider each of these three non-trivial terms one-by-one.

Let first $i \in [1,3N_p]$ and  $j,k \in [3N_p+1,6N_p]$. Then, we obtain the condition
\begin{align*}
\sum_{l=1}^{3N_p} \left( \frac{\partial \mathcal{J}_{12}(\Xib)_{ij}}{\partial \Xi_l} (\mathcal{J}_{12})_{lk}(\Xib)+ \frac{\partial \mathcal{J}_{21}(\Xib)_{ki}}{\partial \Xi_l} (\mathcal{J}_{12})_{lj}(\Xib) \right)=0.
\end{align*}
Inserting the expressions of the terms of the Poisson matrix, we get
\begin{align*}
\sum_{l=1}^{3N_p} \left( \frac{\partial(\NN(\Xib)^\top \MM_m^{-1})_{ij}}{\partial \Xi_l} (\NN(\Xib)^\top \MM_m^{-1})_{lk}
+ \frac{\partial(-\MM_m^{-1} \NN(\Xib))_{ki}}{\partial \Xi_l} (\NN(\Xib)^\top \MM_m^{-1})_{lj}\right).
\end{align*}
Since $\NN$ is a block diagonal matrix, the terms are only non-zero if all four indices belong to the same particle. Since $\MM_m$ is diagonal and has the same entry for each component of one particle, we can leave out this factor. Using the definition of the transposed inverse Jacobian matrix, $\NN_{ij} = \frac{\partial \Xib_j}{\partial \Xb_i}$, we are left with the following expression
\begin{equation*}
\sum_{l=1}^{3N_p}\left( \frac{\partial }{\partial \Xi_l}\frac{\partial \Xi_i}{ \partial X_j} \frac{\partial \Xi_l}{\partial X_k} - \frac{\partial }{\partial \Xi_l}\frac{\partial \Xi_i}{ \partial X_k} \frac{\partial \Xi_l}{\partial X_j}\right) = \frac{\partial^2 \Xi_i}{\partial X_k \partial X_j} -  \frac{\partial^2 \Xi_i}{\partial X_j \partial X_k} = 0,
\end{equation*}
where we have used the chain rule in the last step.

Next, let  $i,j,k \in [3N_p+1,6N_p]$. This yields the following expression to show
\begin{align*}
\sum_{l=1}^{3N_p} \left( \frac{\partial \mathcal{J}_{22}(\Xib)_{ij}}{\partial \Xi_l} \mathcal{J}_{12}(\Xib)_{lk}+ \frac{\partial \mathcal{J}_{22}(\Xib)_{jk}}{\partial \Xi_l} \mathcal{J}_{12}(\Xib)_{li}+\frac{\partial \mathcal{J}_{22}(\Xib)_{ki}}{\partial \Xi_l} \mathcal{J}_{12}(\Xib)_{lj} \right)=0 .
\end{align*}
With the expressions of the Poisson matrix we obtain
\begin{align*}
&\sum_{l=1}^{3N_p}\left( \frac{\partial(\MM_q \NN(\Xib) \tBB(\Xib) \NN(\Xib)^\top )_{ij}}{\partial \Xi_l} (\NN(\Xib)^\top)_{lk}\right.\\
&+ \frac{\partial(\MM_q \NN(\Xib) \tBB(\Xib) \NN(\Xib)^\top )_{jk}}{\partial \Xi_l} (\NN(\Xib)^\top)_{li}
\\
&+\left.\frac{\partial(\MM_q \NN(\Xib) \tBB(\Xib) \NN(\Xib)^\top )_{ki}}{\partial \Xi_l} (\NN(\Xib)^\top)_{lj} \right). 
\end{align*}
Since $\MM_m$ is a diagonal matrix, each term contains $(\MM_m^{-1})_{ii}(\MM_m^{-1})_{jj}(\MM_m^{-1})_{kk}$ which therefore is left out. Moreover, both $\NN$ and $\tBB$ are block-diagonal so the only case in which the terms are non-zero is when all four indices belong to the same particle. Let us denote the corresponding particle index by $p$ and introduce $\mu, \nu, \sigma \in \{1,2,3\}$ as $i-3N_p=3(p-1)+\mu$, $j-3N_p=3(p-1)+\nu$, $k-3N_p=3(p-1)+\sigma$. In this case, we also have $(\MM_q)_{ii} = (\MM_q)_{jj} =(\MM_q)_{kk}$ so that we can leave out this matrix as well. At last, we use the definition of $\NN^\top$ and the following identity for the generic block $\tilde B(\xib_p)$ of $\tBB$, 
\begin{align*}
\hat B_h=\mathbb{I}\times \Bb_h=\mathbb{I}\times \frac{DF}{J_F}\tBb_h= N \left((N^\top ) \times \tBb_h \right)= N \hat{ \tilde B}_h N^\top.
\end{align*}
This leaves us with
\begin{align*}
&\sum_{\eta=1}^3 \left(\frac{\partial \hat B_h(\xb_p)_{\mu  \nu}}{\partial \xi_{p,\eta}} \frac{\partial \xi_{p,\eta}}{\partial x_{p,\sigma}}
 +\frac{\partial \hat B_h(\xb_p)_{\nu  \sigma}}{\partial \xi_{p,\eta}} \frac{\partial \xi_{p,\eta}}{\partial x_{p,\mu}} 
 + \frac{\partial \hat B_h(\xb_p)_{\sigma  \mu}}{\partial \xi_{p,\eta}} \frac{\partial \xi_{p,\eta}}{\partial x_{p,\nu}}  
\right) \\
&=  \left(\frac{\partial \hat B_h(\xb_p)_{\mu\nu}}{\partial x_{p,\sigma}}
+ \frac{\partial \hat B_h(\xb_p)_{\nu\sigma}}{\partial x_{p,\mu}}
+\frac{\partial \hat B_h(\xb_p)_{\sigma\mu}}{\partial x_{p,\nu}} \right).
\end{align*}
If $\mu=\nu=\sigma$, we get the diagonal terms that are zero and if two indices are the same, say $\mu=\nu$, we get $\frac{\partial \hat B_h(\xb_p)_{\mu\sigma}}{\partial x_{p,\mu}}
+\frac{\partial \hat B_h(\xb_p)_{\sigma\mu}}{\partial x_{p,\mu}} = 0$ due to the antisymmetry of $\hat B_h$. Lastly, if the three indices are all different, we get
\begin{equation*}
\pm \left(\frac{\partial \hat B_h(\xb_p)_{12}}{\partial x_{p,3}}
+ \frac{\partial \hat B_h(\xb_p)_{23}}{\partial x_{p,1}}
+\frac{\partial \hat B_h(\xb_p)_{31}}{\partial x_{p,2}}\right) = \pm \div \Bb_h(\xb_p).
\end{equation*}
Since $\div \Bb_h = 0$ is guaranteed over time by the construction of the discrete de Rham complex when it is initially satisfied, this is also zero.

Finally, we consider the case that  $i,j \in [3N_p+1, 6N_p]$,  $k\in [6N_p+1, 6N_p+3N_1]$ 
\begin{align*}
\sum_{l=1}^{3N_p} \left( \frac{\partial \mathcal{J}_{23}(\Xib)_{ik}}{\partial \Xi_l} \mathcal{J}_{12}(\Xib)_{lj}+ \frac{\partial \mathcal{J}_{32}(\Xib)_{kj}}{\partial \Xi_l} \mathcal{J}_{12}(\Xi)_{li}\right)+\sum_{A=1}^{3N_2}\frac{\partial \mathcal{J}_{22}(\tbb)_{ij}}{\partial \tb_A} (\mathcal{J}_{43})_{Ak} =0. 
\end{align*}
With the expressions of the Poisson matrix, we obtain
\begin{align*}
&\sum_{l=1}^{3N_p} \left( \frac{\partial (\MM_m^{-1}\MM_q (\NN\tLaBB^1)(\Xib) \tM_1^{-1})_{ik}}{\partial \Xi_l} (\NN(\Xib)^\top \MM_m^{-1})_{lj}\right.\\
&+ \left.\frac{\partial(-\tM_1^{-1} (\NN \tLaBB^1)(\Xib)^\top \MM_q \MM_m^{-1})_{kj}}{\partial \Xi_l} (\NN(\Xib)^\top \MM_m^{-1})_{li} \right)
\\
&+\sum_{A=1}^{3N_2} \frac{\partial(\MM_m^{-1}\MM_q \NN(\Xib) \tBB(\Xib,\tbb) \NN(\Xib)^\top \MM_m^{-1})_{ij}}{\partial \tb_A} (-\C\tM_1^{-1})_{Ak}=0. 
\end{align*}
We contract this with $\MM_m$ for indices $i,j$, $\tM_1$ on index $k$ and $\MM_q^{-1}$ on index $i$,
\begin{align*}
&\sum_{l=1}^{3N_p} \left(\frac{\partial ((\NN\tLaBB^1)(\Xib) )_{ik}}{\partial \Xi_l} (\NN(\Xib)^\top)_{lj}
- \frac{\partial((\NN \tLaBB^1)(\Xib)^\top \MM_q )_{kj}}{\partial \Xi_l} ( \NN(\Xib)^\top\MM_q^{-1})_{li}\right)\\
&=\sum_{A=1}^{3N_2}\frac{\partial ( \NN(\Xib) \tBB(\Xib,\tbb) \NN(\Xib)^\top)_{ij}}{\partial \tb_A} (\C)_{Ak}.  
\end{align*} This is possible because $\MM_m,\MM_q^{-1}$ and $\tM_1$ are constant, symmetric and positive definite. 
Moreover, we have again that $i$ and $j$ must belong to the same particle due to the block-diagonal structure of the terms. Therefore, we can also contract $\MM_q$ on index $j$ and $\MM_q^{-1}$ on index $i$. Let us introduce again the corresponding particle index $p$ and $\mu,\nu \in\{1,2,3\}$ such that $i-3N_p=3(p-1) + \mu$ and $j-3N_p=3(p-1)+ \nu$. For these index combinations, the sum over the particle positions breaks down to 
\begin{align*}
&\sum_{l=1}^{3N_p} \left(\frac{\partial ((\NN\tLaBB^1)(\Xib) )_{ik}}{\partial \Xib_l} (\NN(\Xib)^\top)_{lj}
- \frac{\partial((\NN \tLaBB^1)(\Xib)^\top \MM_q )_{kj}}{\partial \Xib_l} ( \NN(\Xib)^\top\MM_q^{-1})_{li}\right)
\\
&=  \frac{\partial \LaB^1_{k,\mu}(\xb_{p})}{\partial x_{p,\nu}} -\frac{\partial \LaB^1_{k,\nu}(\xb_{p})}{\partial x_{p,\mu}}
\\
&= \left\{ \begin{matrix}
0 & \text{if } \mu = \nu\\
(\curl \LaB^1(\xb_{p}))_{\sigma k} & \text{ if } (\mu, \nu, \sigma) \text{ cyclic permutation of (1,2,3) }\\
-(\curl \LaB^1(\xb_{p}))_{\sigma k} & \text{ if } (\mu, \nu, \sigma) \text{ non-cyclic permutation of (1,2,3) }\\
\end{matrix} 
\right. ,
\end{align*}
where we used $\NN \tLaBB^1 = \LaBB^1$ and the chain rule in the first equality.
For the derivative with respect to $\tbb$, we use expression \eqref{block_of_IxB} for the block of $\BB$ belonging to particle $p$ to find
\begin{align*}
\sum_{A=1}^{3N_2} \frac{\partial \hat{\tilde B}_h(\xib_p)}{\partial \tb_A} = \sum_{A=1}^{N_2} \hat{ \tLaB}_A(\xib_p),
\end{align*}
where
\begin{align*}
\hat{ \tLaB}_A(\xib_p)= \begin{pmatrix}
0&\tLa_{A}^{2,3}(\xib_p)&-\tLa_{A}^{2,2}(\xib_p)\\
-\tLa_{A}^{2,3}(\xib_p)&0&\tLa_{A}^{2,1}(\xib_p)\\ \tLa_{A}^{2,2}(\xib_p)&-\tLa_{A}^{2,1}(\xib_p)&0
\end{pmatrix}.
\end{align*}
It holds $\hat{\LaB}_A=N \hat{ \tLaB}_A N^\top$ in the same way as $\hat{\tB}_h=N \hat{ \tB}_h N^\top$. Hence, we get
\begin{align*}
&\sum_{A=1}^{3N_2}\frac{\partial ( \NN(\Xib) \tBB(\Xib,\tbb) \NN(\Xib)^\top)}{\partial \tb_A}   = \sum_{A=1}^{N_2} \hat{\LaBB}_A(\Xb).
\end{align*}
Now, for $i-3N_p=3(p-1) + \mu$ and $j-3N_p=3(p-1)+ \nu$, we need the component $(\mu,\nu)$ which is zero if $\mu=\nu$ and $\LaB_{A,\sigma}^{2}$ if $(\mu, \nu, \sigma)$ is a cyclic permutation of $(1,2,3)$ (or the negative if the permutation is non-cyclic). This yields 
\begin{align*}
&\sum_{A=1}^{3N_2}\frac{\partial ( \NN(\Xib) \tBB(\Xib,\tbb) \NN(\Xib)^\top)_{ij}}{\partial \tb_A} (\C)_{Ak} = \sum_{A=1}^{N_2} \hat{\LaBB}_A(\Xb)_{ij} (\C)_{Ak}
\\
&=\left\{
\begin{matrix}
0 &  \text{ if } \mu = \nu \\
(\LaB^2(\xb_p) \C)_{\sigma k } & \text{if } (\mu, \nu, \sigma) \text{ cyclic permutation of (1,2,3) }\\
-(\LaB^2(\xb_p) \C)_{\sigma k} & \text{  if } (\mu, \nu, \sigma) \text{ non-cyclic permutation of (1,2,3) }
\end{matrix}
 \right. .
\end{align*}
Hence, the term vanishes due to the de Rham sequence properties of the basis.
\end{proof}

\subsection{Discrete Casimir invariants}
One class of functionals that are conserved over time in Hamiltonian systems are so-called Casimir invariants, functionals that Poisson commute with all other funtionals. In this section, we consider the discrete Casimir invariants of our discrete Poisson structure  \eqref{equations of motion}, i.e.~functions $C(\ub)$ of our discrete dynamic variables $\ub = (\Xib, \Vb, \teb, \tbb)$ that satisfy
 \begin{align*}
\{C,F\}=0 \Leftrightarrow \mathbb{J}(\ub) DC(\ub)=0 \text{ for all } F(\ub).
\end{align*}
First, we derive a general form for such discrete Casimir invariants and, second, we show that the divergence constraints \eqref{divconstraints} are such discrete Casimir invariants and, hence, conserved over time in our discretisation.

\begin{prop}\label{prop:discrete_casimir_general}
Let $C(\ub)$ be a discrete Poisson invariant of the system \eqref{equations of motion} with Poisson matrix \eqref{Poissonmatrix}. Then, there exist $\bar{\eb} \in \R^{N_0}$ and $\bar{\bb} \in \R^{N_3}$ such that
\begin{equation}\label{eq:discrete_casimir_general}
C(\ub) = \bar{\eb}^\top(\tLaBB^0(\Xib)^\top \MM_q \mathbbm{1}_{N_p}+  \G^\top \tM_1 \teb) + \bar{\bb}^\top \D \tbb.
\end{equation}
\end{prop} 
\begin{proof}
Let us consider the equation $\mathbb{J}(\ub) DC(\ub)=0$ line by line. The first line reads
\begin{equation*}
\NN(\Xib)^\top \MM_m^{-1}  D_{\Vb} C = 0.
\end{equation*}
Therefore, $C$ must be independent of $\Vb$. Next, we consider the third line, already assuming $D_{\Vb}C = 0$, which yields
\begin{align*}
\tM_1^{-1} \C^\top  D_{\tbb} C = 0.
\end{align*}
Hence, it follows that $D_{\tbb}C \in \ker(\C^\top)$. Due to the complex property of our de Rham sequence, there exist a $\bar{\bb}\in \R^{N_3}$ such that $D_{\tbb}C= \D^\top \bar{\bb}$.
Analogously, the fourth line,
\begin{equation*}
\C \tM_1^{-1}   D_{\teb} C = 0, \text{ i.e.~} \tM_1^{-1} D_{\teb}C \in \ker(\C),
\end{equation*}
yields due to the complex property that there exists an  $\bar{\eb}\in \R^{N_0}$ such that $D_{\teb}C= \tM_1 \G \bar{\eb}$. Finally, the second line of the Poisson matrix yields the following expression for $D_{\Xib} C $,
\begin{equation}
D_{\Xib} C = \MM_q  \tLaBB^1(\Xib) \tM_1^{-1} D_{\teb}C= \MM_q  \tLaBB^1(\Xib) \G \bar{\eb}= \MM_q \grad \tLaBB^0(\Xib) \bar{\eb},
\end{equation}
where we used again the complex property for the last equality. Putting everything together, we get the general form \eqref{eq:discrete_casimir_general} of a discrete Casimir invariant.
\end{proof}
As a consequence the divergence constraints \eqref{divconstraints} are conserved over time.

\begin{cor}
The discrete magnetic Gauss law, $\D \tbb=0$, is conserved over time if it is satisfied initially.
\end{cor}
\begin{proof}
This follows immediately from Proposition \ref{prop:discrete_casimir_general} setting $\bar{\eb} = \mathbb{0}_{N_0}$ and $\bar{\bb}=\mathbbm{1}_{N_3}$, since this leads to the discrete Casimir $\D \tbb$.
\end{proof}
\begin{cor}
The discrete electric Gauss law, $\G^\top \tM_1 \teb -\MM_q \tLaBB^0(\Xib)^\top \mathbb{1}_{N_p}=0$, is conserved over time if it is satisfied initially.
\end{cor}
\begin{proof}
This follows immediately from Proposition \ref{prop:discrete_casimir_general} setting $\bar{\eb} = \mathbbm{1}_{N_0}$ and $\bar{\bb}=\mathbb{0}_{N_3}$, since this leads to the discrete Casimir $\G^\top \tM_1 \teb -\MM_q \tLaBB^0(\Xib)^\top \mathbb{1}_{N_p}$.
\end{proof}

\section{Time discretisation of the equations of motion}\label{sec:time}
For the GEMPIC method on Cartesian grids, two structure-preserving time propagation schemes that exploit the form 
\eqref{eq:equations_of_motion_poisson} have been proposed: In \cite{kraus2016gempic}, the discrete Hamiltonian \eqref{disHam} is split into five  parts $ \tilde{\mathcal{H}}_{h,i}$, $i=1,\ldots, 5$ so that each part of the equations $\dot \tub= \{\tub, \tilde{\mathcal{H}}_{h,i}\}$ yields explicit equations of motion and the discretisation preserves Gauss' law. This splitting was first proposed in \cite{he2015hamiltonian}. A second ansatz is to split system \eqref{eq:equations_of_motion_poisson} by an antisymmetric splitting of the Poisson matrix $\mathbb{J}$. Then, the resulting subsystems can be solved by a discrete gradient method yielding an energy-preserving time stepping scheme. In \cite{kormann2019}, two schemes are constructed this way: a semi-implicit scheme that does not preserve the electric Gauss law and a fully discrete scheme that preserves Gauss' law. The discrete gradient schemes readily extend to the curvilinear case as we will show in sections \ref{sec:time_avf} and \ref{sec:time_disgrad}. However, an explicit Hamiltonian splitting can no longer be constructed, since the coordinate directions do not decouple for non-diagonal coordinate transformations. Instead, we will construct a semi-implicit splitting that preserves Gauss' law in section \ref{sec:time_hs}.

\subsection{Charge conserving splitting}\label{sec:time_hs}
In this section, we consider a Hamiltonian splitting as in \cite{kraus2016gempic}, however, we only split into three parts,
\begin{align*} \tilde{\mathcal{H}}_h=\tilde{\mathcal{H}}_p+\tilde{\mathcal{H}}_E+\tilde{\mathcal{H}}_B
\end{align*}
with 
\begin{align*}
\tilde{\mathcal{H}}_p= \frac{1}{2}\Vb^\top \MM_p \Vb,\ \tilde{\mathcal{H}}_E=\frac{1}{2}\teb^\top \tM_1 \teb, \ \tilde{\mathcal{H}}_B=\frac{1}{2}\tbb^\top \tM_2 \tbb.
\end{align*}
Thus, we obtain the three subsystems
\begin{align*}
\dot{\mathbf{u}}=\{\mathbf{u},\tilde{\mathcal{H}}_p \},\  \dot{\mathbf{u}}=\{\mathbf{u},\tilde{\mathcal{H}}_E \}, \ \dot{\mathbf{u}}=\{\mathbf{u},\tilde{\mathcal{H}}_B \}.
\end{align*}
The subsystems for $\tilde{\mathcal{H}}_E$ and $\tilde{\mathcal{H}}_B$ are solved exactly and, then, evaluated at the discrete time steps $t^n=n \Delta t$. Let us denote $\tub(t^n)=:\tub^n$. 
 
For $\tilde{\mathcal{H}}_E$, the equations of motion are
\begin{equation}\label{HE}
\begin{aligned}
\dot{\Xib}&=0, &&\dot{\Vb}=\MM_p^{-1}\MM_q \NN (\xi) \tLaBB^1(\Xib) \teb,
\\
\dot{\teb}&=0,  &&\dot{\tbb}=-\C \teb
\end{aligned}
\end{equation}
and the time discrete version reads
\begin{align*}
\Xib^{n+1}&=\Xib^n, &&\Vb^{n+1}=\Vb^n+\Delta t\MM_p^{-1}\MM_q \NN (\Xib^n) \tLaBB^1(\Xib^n) \teb^n,
\\
\teb^{n+1}&=\teb^n, && \tbb^{n+1}=\tbb^n-\Delta t \C \teb^n.
\end{align*}
For $\tilde{\mathcal{H}}_B$, we get 
\begin{equation}
\label{HBe}
\dot{\Xib}=0, \quad \dot{\Vb}=0, \quad \tM_1\dot{\teb}=\C^\top \tM_2 \tbb, \quad \dot{\tbb}=0,
\end{equation}
which leads to the discretisation
\begin{align*}
\Xib^{n+1}=\Xib^n, \quad \Vb^{n+1}=\Vb^n, \quad
\tM_1 \teb^{n+1}=\tM_1 \teb^n+\Delta t \C^\top \tM_2 \tbb^n, \quad 
\tbb^{n+1}=\tbb^n.
\end{align*}
For $\tilde{\mathcal{H}}_p$, we obtain the following equations 
\begin{equation}\label{Hp}
\begin{aligned}
\dot{\Xib}&=\NN(\Xib)^\top\Vb, &&\dot{\Vb}= \MM_p^{-1}\MM_q \NN(\Xib) \tBB(\Xib,\tbb) \NN(\Xib)^\top\Vb ,
\\
\tM_1 \dot{\teb}&= -\tLaBB^1(\Xib)^\top \NN(\Xib)^\top \MM_q  \Vb, &&\dot{\tbb}=0.
\end{aligned}
\end{equation}
Here, we get the analytic solution
\begin{align*}
\Xib(\Delta t)&=\Xib(0)+ \int^{\Delta t}_0 \NN(\Xib(t))^\top \Vb(t) \ \du t,\\
\Vb(\Delta t)&=\Vb(0)+\MM_p^{-1}\MM_q \int^{\Delta t}_0  \NN(\Xib(t))\tBB(\Xib(t),\tbb(0)) \NN(\Xib(t))^\top \Vb(t)  \ \du t,\\
\tM_1 \teb(\Delta t)&= \tM_1 \teb(0)-\int^{\Delta t}_0   \tLaBB^1(\Xib(t))^\top \NN(\Xib(t))^\top \MM_q  \Vb(t)  \ \du t\\
\tbb(\Delta t)&=\tbb(0).
\end{align*}
This system is implicit in the particle coordinates $(\Xib, \Vb)$ but decouples between different particles. It is not possible to solve the resulting $6\times 6$ systems explicitly. Therefore, the kinetic energy part was further split into the three directions in \cite{kraus2016gempic}. Such a splitting yields explicit equations only if the Jacobi matrix of the coordinate transformation is diagonal and constant. Since this is generally not true, we keep the kinetic part together.

In order to resolve the non-linearity caused by the dependence of $\NN$ on $\Xib$, we need to introduce some approximation, which should conserve the symplectic structure. In \cite{kormann2019} it has been shown that a Gauss-conserving discretisation can be obtained when using the same constant velocity for both the position and the current update. 
%
That is why, we solve the particle equations with the symplectic midpoint method in a fix point iteration using a predictor-corrector scheme and then, compute the current for the update of the electric field with a line integral for $\tLaBB^1(\Xib(t))$ and the velocity from the position update. This results in the following system,
\begin{subequations}\label{hs}
\begin{align}\label{hs1}
\Xib^{n+1}=&\Xib^n+ \Delta t \NN^\top\left(\overline{\Xib}\right)\overline{\Vb},
\\\notag
\Vb^{n+1}=&\Vb^{n}+ \Delta t \MM_m^{-1}\MM_q \NN\left(\overline{\Xib}\right) \tBB\left(\overline{\Xib},\tbb^n\right) \NN^\top\left(\overline{\Xib}\right)\overline{\Vb},
\\\label{hs3}
\tM_1 \teb^{n+1}= &\tM_1 \teb^n-\int^{t^{n+1}}_{t^n}   \tLaBB^1(\Xib(\tau))^\top \du\tau \MM_q \NN^\top\left(\overline{\Xib}\right)\overline{\Vb},
\\ \notag
\tbb^{n+1}=&\tbb^n,
\end{align}
\end{subequations}
where $\overline{\Xib}=\frac{\Xib^n+\Xib^{n+1}}{2}, \overline{\Vb}=\frac{\Vb^n+\Vb^{n+1}}{2}$ and $\Xib(\tau)= \frac{(t^{n+1}-\tau)\Xib^n+(\tau-t^n)\Xib^{n+1}}{\Delta t}$.

\begin{prop}\label{gaussconservation} For the proposed splitting, 
Gauss' law is preserved over time if it is satisfied initially and \eqref{Hp} is discretised as in \eqref{hs}.
\end{prop}
\begin{proof}
First, we identify the two splitting steps in which the electric field is changed. In $\mathcal{H}_B$, the update of the electric field \eqref{HBe} multiplied by $\G^\top$ stays constant due to the discrete complex property \eqref{discomp}. For $\mathcal{H}_p$, we multiply \eqref{hs3} with $\G^\top$ and plug in the position formula \eqref{hs1} and use that $\frac{\du\Xib}{\du\tau}=\frac{\Xib^{n+1}-\Xib^n}{\Delta t}$ is constant in time,
\begin{align*}
\G^\top\tM_1\teb^{n+1}-\G^\top\tM_1\teb^{n}&=- \int_{t^n}^{t^{n+1}} \G^\top \tLaBB^1(\Xib(\tau))^\top  \du\tau \MM_q  \frac{\Xib^{n+1}-\Xib^n}{\Delta t}\\
&=- \int_{t^n}^{t^{n+1}} \MM_q \G^\top \tLaBB^1(\Xib(\tau))^\top \frac{\du\Xib(\tau)}{\du\tau} \du\tau.
\end{align*}
At last, we use the chain rule , $\frac{\du \tLaBB^0(\Xib(\tau))^\top}{\du\tau} \mathbb{1}_{N_p}= \G^\top \tLaBB^1(\Xib(\tau))^\top \frac{\du\Xib(\tau)}{\du\tau} $, and obtain 
\begin{align*}
\G^\top\tM_1\teb^{n+1}-\G^\top\tM_1\teb^{n}&= -\int_{t^n}^{t^{n+1}} \MM_q \frac{\du \tLaBB^0(\Xib(\tau))^\top}{\du\tau} \mathbb{1}_{N_p}  \du\tau\\
&= -\left(\MM_q  \tLaBB^0(\Xib^{n+1})^\top \mathbb{1}_{N_p}-\MM_q  \tLaBB^0(\Xib^{n})^\top \mathbb{1}_{N_p}\right).
\end{align*}
\end{proof}
Note that the source-free Maxwell equations are solved in two different splitting steps which causes a restriction on the time step (cf.~\cite[Appendix A.2]{kormann2019}. For the simulation results of the Hamiltonian splitting we use the acronym HS.

\subsection{Energy conserving antisymmetric splitting}\label{sec:time_avf}
Next, we consider energy-conserving time discretisations constructed as discrete gradients \cite{quispel1996discrete}. 

\begin{theorem}
Let us consider a system of ordinary differential equations of the form
\begin{align*}
\dot \tub= \mathcal{J}(\tub) D\tilde{\mathcal{H}}_h(\tub)
\end{align*}
with a skew-symmetric matrix $\mathcal{J}$. Then, the discrete gradient discretisation of the form
\begin{align*}
\frac{\tub^{n+1}-\tub^n}{\Delta t}= \bar{\mathcal{J}}(\tub^{n+1},\tub^n) \overline{D\tilde{\mathcal{H}}}_h(\tub^{n+1},\tub^n)
\end{align*}
is energy conserving if $\bar{\mathcal{J}}(\tub^{n+1},\tub^n)$ is skew-symmetric.
\end{theorem}
\begin{proof}
The energy variation in one time step is defined as
\begin{align*}
\tilde{\mathcal{H}}_h^{n+1}-\tilde{\mathcal{H}}_h^{n}&= D\tilde{\mathcal{H}}_h(\tub^{n+1},\tub^n)^\top \left(\tub^{n+1}-\tub^n\right).
\end{align*}
Now, we insert the discretisation from above and use the skew-symmetry of $\bar{\mathcal{J}}$,
\begin{align*}
\tilde{\mathcal{H}}_h^{n+1}-\tilde{\mathcal{H}}_h^{n}&= D\tilde{\mathcal{H}}_h(\tub^{n+1},\tub^n)^\top \left(\Delta t\bar{\mathcal{J}}(\tub^{n+1},\tub^n)(\tub) D\tilde{\mathcal{H}}_h(\tub^{n+1},\tub^n) \right)\\
&=-\Delta t  D\tilde{\mathcal{H}}_h(\tub^{n+1},\tub^n)^\top \bar{\mathcal{J}}(\tub^{n+1}, \tub^n) D\tilde{\mathcal{H}}_h(\tub^{n+1},\tub^n)=0.
\end{align*}
\end{proof}
Several ways to construct discrete gradients have been proposed in the literature \cite{Itoh85,Gonzalez96,celledoni2012preserving}. However, in our case the Hamiltonian is quadratic---and $D\tilde{\mathcal{H}}_h$ linear---so that all methods simplify to the second order midpoint rule.  Hence, this leaves us with the choice how to discretise $\mathcal{J}$. Moreover, we follow \cite{kormann2019} and split the discrete Poisson matrix, keeping its skew-symmetry in each subsystem. 
We split the Poisson matrix into four antisymmetric parts and obtain the following subsystems,
\begin{equation}
\begin{aligned}\label{antisym}
&\text{ system 1: } 
\dot{\Xib}=\NN(\Xib)^\top\Vb,
\\
&\text{ system 2: }
\dot{\Vb}=\MM_m^{-1}\MM_q \NN(\Xib) \tBB(\Xib,\tbb) \NN(\Xib)^\top \Vb,\\
&\text{ system 3: }
 \dot{\tbb}=-\C \teb, \ \tM_1 \dot{\teb}=\C^\top \tM_2 \tbb,
\\
&\text{ system 4: }
\dot{\Vb}=\MM_m^{-1}\MM_q \NN (\Xib) \tLaBB^1(\Xib) \teb, \
\tM_1 \dot{\teb} = - \tLaBB^1(\Xib)^\top \NN(\Xib)^\top \MM_q  \Vb.
\end{aligned}
\end{equation}
In the first system, the element of the Poisson matrix $\NN(\Xib)$ is changing over time and needs to be approximated. We use a Crank-Nicholson method to maintain the second order accuracy and solve the system iteratively with a predictor-corrector scheme,  
\begin{align*}
\Xib^{n+1}=&\Xib^n+\Delta t \frac{\NN(\Xib^n)^\top+ \NN(\Xib^{n+1})^\top}{2} \Vb^n.
\end{align*}
Note that the system is block-diagonal and, hence, only coupling the positions of one particle at a time.

In the other three systems, the Poisson matrix is constant over time and we only use the midpoint rule to discretise the $D\tilde{\mathcal{H}}_h$ part. So, the equation for the second system reads
\begin{equation}
\begin{aligned}\label{vavf}
\left(\mathbb{I}-\frac{\Delta t}{2} \MM_m^{-1}\MM_q \NN\tBB \NN^\top \right)\Vb^{n+1}&=\left(\mathbb{I}+\frac{\Delta t}{2} \MM_m^{-1}\MM_q \NN \tBB \NN^\top\right) \Vb^{n},
\end{aligned}
\end{equation}
where for every particle the inverse of the $3\times3$ Matrix on the left hand side can be exactly calculated.

With the same method, the matrix form of the equations for system 3 reads
\begin{align*}
\begin{pmatrix}
\tM_1 & -\frac{\Delta t}{2} \C^\top  \tM_2\\
+\frac{\Delta t}{2} \C& \mathbb{I}
\end{pmatrix} \begin{pmatrix}
\teb^{n+1}\\
\tbb^{n+1}
\end{pmatrix}= \begin{pmatrix}
\tM_1 & +\frac{\Delta t}{2} \C^\top  \tM_2\\
-\frac{\Delta t}{2} \C& \mathbb{I}
\end{pmatrix} \begin{pmatrix}
\teb^{n}\\
\tbb^{n}
\end{pmatrix}.
\end{align*}
With the Schur complement $S=\tM_1+\frac{\Delta t^2}{4} \C^\top \tM_2 \C$, we get the following decoupled system,
\begin{equation}
\begin{aligned}\label{ebavf}
\eb^{n+1}&=S^{-1}\left((\tM_1-\frac{\Delta t^2}{4} \C^\top \tM_2 \C)\eb^n+\Delta t \C^\top \tM_2 \bb^n \right),\\
\bb^{n+1}&=\bb^n-\frac{\Delta t}{2} \C(\eb^n+\eb^{n+1}).
\end{aligned}
\end{equation}
Finally, system 4 is discretised in matrix-vector notation as
\begin{align*}
A_{-} \begin{pmatrix}
\Vb^{n+1}\\
\teb^{n+1}
\end{pmatrix} = A_{+} \begin{pmatrix}
\Vb^{n}\\
\teb^{n}
\end{pmatrix}, \text{ with } A_{\pm}= \begin{pmatrix}
\mathbb{I}&  \pm\frac{\Delta t}{2}\MM_m^{-1}\MM_q \NN \tLaBB^1\\
\mp \frac{\Delta t}{2} (\tLaBB^1)^\top \NN^\top \MM_q   &\tM_1 
\end{pmatrix} 
\end{align*}
We use the Schur complement $S=\tM_1+\frac{\Delta t^2}{4}\MM_q^2 \MM_m^{-1}(\tLaBB^1)^\top \NN^\top \NN\tLaBB^1$ to decouple the equations for $\Vb$ and $\eb$, which yields the following system that is explicit in the particles and linearly implicit in the electric field
\begin{align*}
\eb^{n+1}&=S^{-1}\left((\tM_1-\frac{\Delta t^2}{4} \MM_q^2 \MM_m^{-1} \MM^\star)\eb^n-\Delta t (\tLaBB^1)^\top \NN^\top  \MM_q \Vb^n \right),\\
\Vb^{n+1}&=\Vb^n+\frac{\Delta t}{2} \MM_m^{-1} \MM_q \NN \tLaBB^1(\eb^n+\eb^{n+1}).
\end{align*}
Here, we introduced the so-called particle mass matrix, $\MM^\star:=(\tLaBB^1)^\top \NN^\top \NN \tLaBB^1$.

When we look at the charge conservation of the system, we notice that the conservation of Gauss' law gets lost, since the current is not computed in the same splitting step as the position update which is pointed out in \cite{kormann2019}. The simulation results of this energy conserving discrete gradient method are labelled as DisGradE.

\subsection{Energy and charge conserving antisymmetric splitting}\label{sec:time_disgrad}

In this section, we change the splitting and solve systems 1 and 4 from the antisymmetric splitting \eqref{antisym} together with the goal of devising a discrete gradient that also preserves Gauss' law. The three subsystems are then given as
\begin{align*}
&\text{ system 1: }
\dot{\Xib}=\NN(\Xib)^\top\Vb,
\dot{\Vb}=\MM_p^{-1}\MM_q \NN (\Xib) \tLaBB^1(\Xib) \teb, \
\tM_1 \dot{\teb} = - \tLaBB^1(\Xib)^\top \NN(\Xib)^\top \MM_q  \Vb,\\
&\text{ system 2: }
\dot{\Vb}=\MM_p^{-1}\MM_q \NN(\Xib) \tBB(\Xib,\tbb) \NN(\Xib)^\top \Vb,\\
&\text{ system 3: }
 \dot{\tbb}=-\C \teb, \ \tM_1 \dot{\teb}=\C^\top \tM_2 \tbb.
\end{align*}
For the first system, we have to construct a discretisation of the partial Poisson matrix that is antisymmetric to maintain the energy conservation. Moreover, we are aiming at an approximation that preserves Gauss law. Both goals can be achieved with the following discretisation,
\begin{subequations}\label{disgrad}
\begin{align}\notag
\frac{\Xib^{n+1}-\Xib^n}{\Delta t}&=\frac{\NN(\Xib^n)^\top+\NN(\Xib^{n+1})^\top}{2} \frac{\Vb^{n+1}+\Vb^n}{2},
\\\label{disgrad2}
\frac{\Vb^{n+1}-\Vb^n}{\Delta t}&=\MM_p^{-1}\MM_q \frac{\NN(\Xib^n)+\NN(\Xib^{n+1})}{2} \frac{1}{\Delta t}\int_{t^n}^{t^{n+1}} \tLaBB^1(\Xib(\tau))  \du\tau \frac{\teb^{n+1}+\teb^n}{2},
\\\label{disgrad3}
\frac{\tM_1 \teb^{n+1}-\tM_1 \teb^n}{\Delta t} &= - \frac{1}{\Delta t} \int_{t^n}^{t^{n+1}} \tLaBB^1(\Xib(\tau))^\top  \du\tau  \frac{\NN(\Xib^n)^\top+\NN(\Xib^{n+1})^\top}{2} \MM_q\frac{\Vb^{n+1}+\Vb^n}{2},
\\\notag
\tbb ^{n+1}&=\tbb ^n.
\end{align}
\end{subequations}
Since the system \eqref{disgrad} is implicit, it has to be solved iteratively: We first loop over the particle position and velocity and then, update the electric field with the computed current. The whole system is looped over in a fix point iteration for the electric field.

The last two systems are still solved as in \eqref{vavf} and \eqref{ebavf}. 
So, we look at the conservation properties of this splitting.
\begin{prop}
The discrete energy of the splitting defined by \eqref{disgrad}, \eqref{vavf} and \eqref{ebavf} is conserved.
\end{prop}
\begin{proof}
As the systems 2 and 3 are still discretised with the discrete gradient method, they trivially conserve the discrete energy. Therefore, we only have to check the discretisation of the first system. The variation of the discrete energy in this splitting step is given by
\begin{align*}
\tilde{\mathcal{H}}_h^{n+1}-\tilde{\mathcal{H}}_h^{n}=
\frac{1}{2} [ (\Vb^{n+1})^\top\MM_p \Vb^{n+1}-(\Vb^n)^\top\MM_p\Vb^n + (\teb^{n+1})^\top\tM_1 \teb^{n+1}-(\teb^n)^\top\tM_1\teb^n ].
\end{align*}
We multiply \eqref{disgrad2} with $ (\Vb^{n+1}+\Vb^n)^\top\MM_p $ to find, after some reordering,
\begin{align*}
 &(\Vb^{n+1})^\top\MM_p \Vb^{n+1}-(\Vb^n)^\top\MM_p\Vb^n=\\ &\left(\int_{t^n}^{t^{n+1}} \tLaBB^1(\Xib(\tau))^\top  \du\tau \frac{\NN(\Xib^n)^\top+\NN(\Xib^{n+1})^\top}{2}\MM_q \frac{\Vb^{n+1}+\Vb^n}{2} \right)^\top (\teb^{n+1}+\teb^n).
\end{align*}
Using \eqref{disgrad3} to express the right hand side, yields
\begin{align*}
(\Vb^{n+1})^\top\MM_p \Vb^{n+1}-(\Vb^n)^\top\MM_p\Vb^n&=-\left( \tM_1 \teb^{n+1}-\tM_1 \teb^n \right)^\top (\teb^{n+1}+\teb^n)\\
&=-\left((\teb^{n+1})^\top\tM_1 \teb^{n+1}-(\teb^n)^\top\tM_1\teb^n\right).
\end{align*}
\end{proof}
\begin{prop}
For the splitting defined by \eqref{disgrad}, \eqref{vavf} and \eqref{ebavf},  Gauss' law is preserved over time if it is satisfied initially and system 1 is discretised as in \eqref{disgrad}.
\end{prop}
\begin{proof}
This can be proved in the same way as Prop.~\ref{gaussconservation} with $\frac{\Xib^{n+1}-\Xib^n}{\Delta t}=$const. in this step.
\end{proof}
The simulation results of this charge and energy conserving discrete gradient method are labelled as DisGradEC.

Let us compare the building blocks of the DisGradEC method to the HS in terms of complexity. Since usually the number of particles is much larger than the number of degrees of freedom for the fields, the most expensive step is the evaluation of the line integral for the current deposition (cf.~\cite[Sec.~5.2.2.]{kormann2019}). However, for DisGradEC, this evaluation needs to be repeated in each nonlinear iteration. Moreover, for the source free Maxwell equations, the computation of the Schur complement for DisGradEC is more expensive than the explicit solution for HS.

The DisGradE scheme treats the source-free Maxwell equations in the same way as DisGradEC. The computationally most expensive part in this case is the assembly of the particle mass matrix. Both the evaluation of the line integral and the particle mass matrix depend to the sixth power on the spline order (cf.~the discussion in \cite{kormann2019}). In the case of DisGradEC, however, the constant depends on the number of cells crossed by the line integral as well as the number of nonlinear iterations, so that a general comparison is not possible.


\section{Numerical experiments}\label{sec:numerics}
We have implemented the 3D3V propagators in curvilinear coordinates as part of the SeLaLib library \cite{selalib} (see \cite[Appendix A]{artigues2019} for some details on the implementation). In this section, we reproduce two numerical test cases from \cite{kraus2016gempic} in a 3D setting in order to validate the code. Additionally, we perform an actual 3D simulation and compare the conservation properties of the different schemes. All numerical simulations are performed for electrons with a neutralising ion background. The particles are loaded with Sobol numbers and sampled uniformly in logical configuration space. In the absence of a coordinate transformation, the mass matrices are diagonal and can thus be inverted in Fourier space (cf.~\cite{kormann2019}). With a coordinate transformation, this is no longer true and we use a conjugate gradient solver instead that we precondition with the Fourier solver for the Cartesian case. The idea to use a direct solver on the Cartesian mesh as preconditioner for an iterative solver on the curvilinear mesh was borrowed from \cite{donatelli2015robust}. Note that this yields a solution to machine accuracy for the Cartesian case. Therefore, we switch off the preconditioner in this case in order to show comparable accuracy in the conservation properties, which depends on the solver tolerance.

\subsection{Coordinate transformation}
We use the following two periodic coordinate transformations for our tests, an orthogonal non-uniform transformation and the Colella transformation defined by the following functions,
\begin{align*}
F_{\text{orth}}(\xib)=\begin{pmatrix}
L\left(\xi_1+ \epsilon \sin(2 \pi\xi_1) \right) \\
L\left(\xi_2+ \epsilon \sin(2\pi \xi_2)\right) \\
L\xi_3
\end{pmatrix}
,\quad F_{\text{Colella}}(\xib)=\begin{pmatrix}
L\left(\xi_1+ \epsilon \sin(2 \pi\xi_1) \sin(2 \pi\xi_2)\right) \\
L\left(\xi_2+ \epsilon \sin(2 \pi\xi_1) \sin(2\pi \xi_2)\right) \\
L\xi_3
\end{pmatrix}.
\end{align*}
We choose $\epsilon < \frac{1}{2\pi}$ so that the inverse Jacobi matrix does not become singular.  Figure \ref{fig:meshes} visualises the $(x,y)$-part of the corresponding meshes for $\epsilon=0.1$.

\begin{figure}
\centering
 \begin{subfigure}{0.4\textwidth}
  \centering
  \includegraphics[width=.6\linewidth]{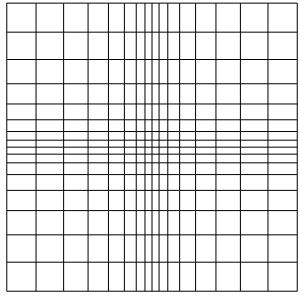}
  \caption{Orthogonal non-uniform mesh}
  \label{fig:orthogonalmesh}
 \end{subfigure}
 \hspace{0.5cm}
 \begin{subfigure}{0.4\textwidth}
  \centering
  \includegraphics[width=.6\linewidth]{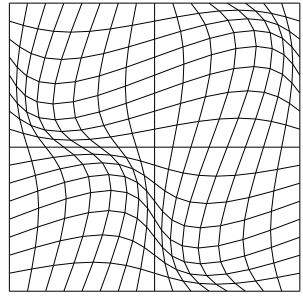}
  \caption{Colella mesh}  
  \label{fig:colellamesh}
 \end{subfigure}
 \caption{Orthogonal and Colella mesh for $\epsilon = 0.1$.}
 \label{fig:meshes}
\end{figure}

\subsection{Weibel instability}
Let us first consider the Weibel instability \cite{weibel1959spontaneously} with the 3D3V initial distribution
\begin{align*}
f(\xb,\vb,t=0)= \left(1+\alpha \cos( \mathbf{k} \cdot \xb)\right)\frac{1}{(2\pi)^\frac{3}{2} v_{th1}v_{th2}^2} \exp \left( -\frac{1}{2} \left( \frac{v_1^2}{v_{th1}^2}+\frac{v_2^2+v_3^2}{v_{th2}^2}\right) \right ),
\end{align*}
where  $\xb \in [0,L]^3, \vb \in \R^3$.
The magnetic field is initially set to $ \Bb(\xb,t=0)= \beta \cos(\mathbf{k} \cdot \xb) \mathbf{e}_3$
and
$\Eb(\xb,t=0)$ is calculated from Poisson's equation.
We choose the parameter as $v_{th1}=0.02/\sqrt{2}, v_{th2}=\sqrt{12}v_{th1},\mathbf{k}=(1.25,0,0)^\top,\alpha=0, L=\frac{2 \pi}{1.25}$ and $\beta=10^{-3}$.
For the numerical resolution, we take $1,600,000$ particles, $32\times4\times2$ grid cells, spline degrees $(3,2,1)$ and for the  iterative solver a tolerance of $10^{-13}$, for the non-linear iteration in DisGradEC a tolerance of  $10^{-12}$. 
These parameters are comparable to the 1D2V settings in \cite{kraus2016gempic}. Only $\beta$ is chosen one magnitude larger so that the initial growth of the magnetic field is higher than the effects caused by the particle noise at the chosen resolution. 

The time step of the Hamiltonian splitting is set to $\Delta t = 0.025$, which is close to the numerical stability limit of these methods. For the discrete gradient schemes, where each time step is considerably more expensive, we increase the time step to $\Delta t = 0.05$. 

\begin{figure}
\centering
 \begin{subfigure}{0.47\textwidth}
  \centering
  \includegraphics[width=\linewidth]{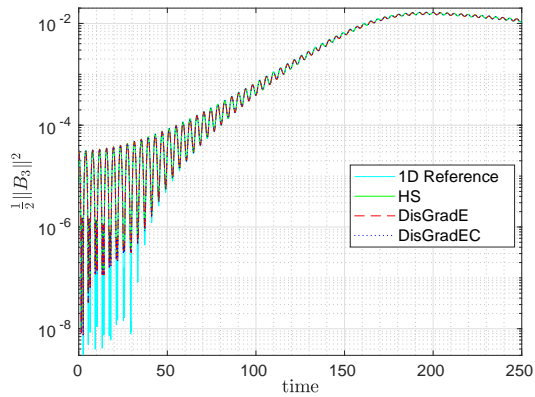}
  \caption{Orthogonal non-uniform mesh}
 \label{fig:magneticfield1}
 \end{subfigure}
 \hfill
 \begin{subfigure}{0.47\textwidth}
  \centering
  \includegraphics[width=\linewidth]{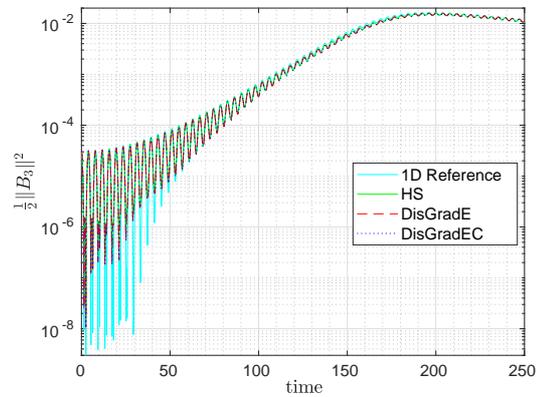}
\caption{Colella mesh}
\label{fig:magneticfield2}
 \end{subfigure}
 \caption{Weibel instability: Growth rate for magnetic field with time step $\Delta t=0.025$ for HS and $\Delta t=0.05$ for DisGradE, DisgradEC and distortion parameter $\epsilon=0.1$.} 
\end{figure}
Figure \ref{fig:magneticfield1} shows the third component of the magnetic field energy as a function of time for different propagators on the orthogonal non-uniform mesh and Figure \ref{fig:magneticfield2} shows the same quantity on the Colella grid. In both cases, a 1D reference run with an explicit Hamiltonian splitting on Cartesian coordinates is given for comparison using the 1D Weibel distribution from \cite{kraus2016gempic} with $\beta=10^{-3}$.

Next, we set the wave vector to $\mathbf{k}=(1.25,1.25,1.25)^\top$ and $\alpha=0.1$ so that we have a perturbation in every $x$-component. 
For the numerical resolution, we take 8 grid cells in every direction and spline degrees $(3,3,3)$. The other parameters remain unchanged.

Now, we look at runs of the Hamiltonian splitting for different distortion parameters $\epsilon$ of the transformation. We take $\epsilon=0$ as a reference and go from $\epsilon=0.01$ up to $\epsilon=0.1$ to study the effect of the coordinate transformation. The time step is taken as $\Delta t=0.01$ to obey the CFL-condition for all choices of $\epsilon$. The initial distribution is sampled in logical coordinates. Hence, the number of particles per cell is approximately constant. The larger the distortion of the grid, the larger cells appear and parts of the domain become more and more underresolved and the quality of the solution decreases. Note that the coordinate transformations which we consider here are artificial with the goal to validate our method. The situation is different when problem-specific coordinate transformations are used in realistic scenarios. 
In Figure \ref{fig:magneticfield3}, we see that for decreasing $\epsilon$ the magnetic field growth rate converges to the scaling case with $\epsilon=0$ which coincides with the run without transformation.

\begin{figure}
\centering
 \begin{subfigure}{0.47\textwidth}
  \centering
  \includegraphics[width=\linewidth]{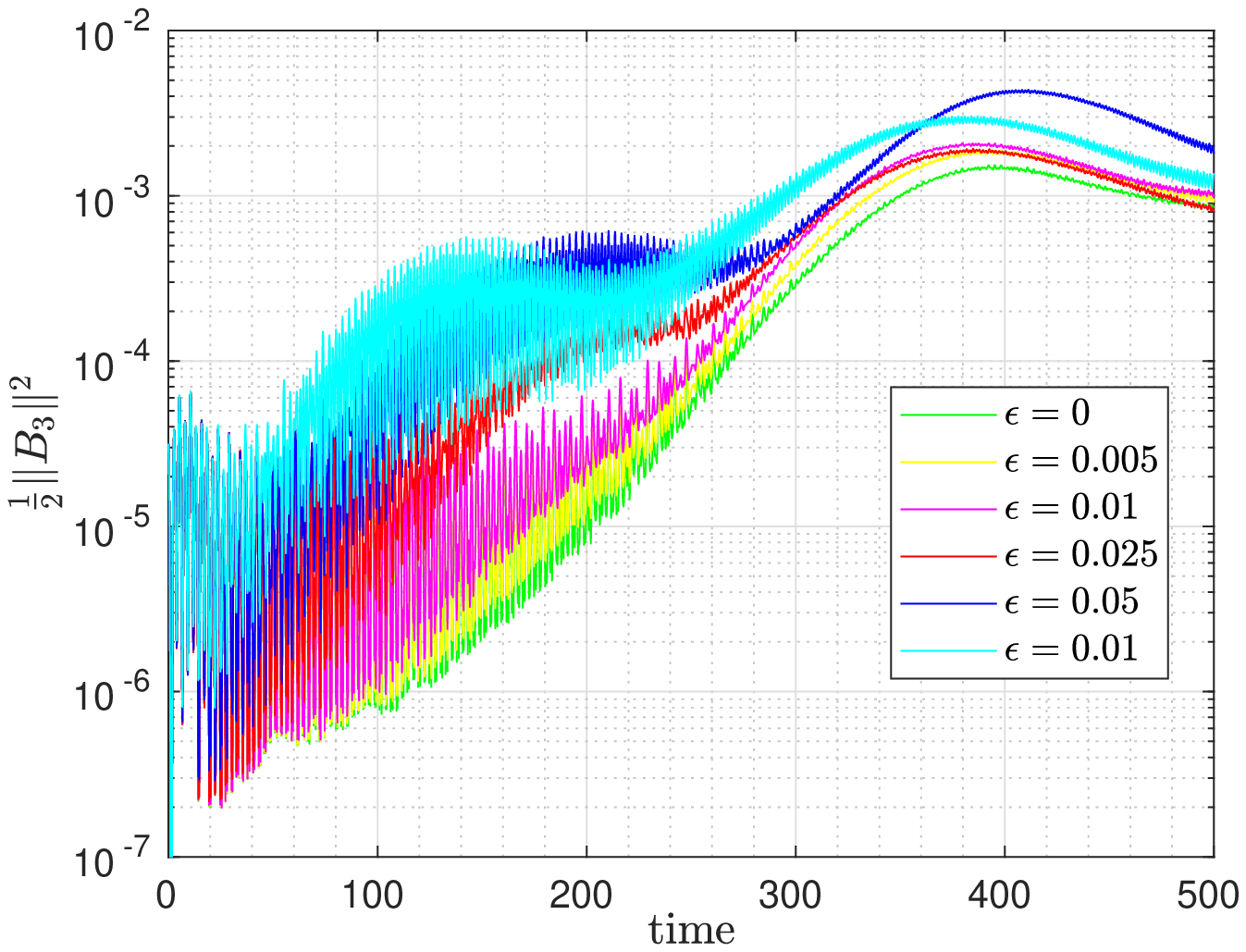}
\caption{Magnetic field energy}
\label{fig:magneticfield3}
 \end{subfigure}
 \hfill
 \begin{subfigure}{0.47\textwidth}
  \centering
  \includegraphics[width=\linewidth]{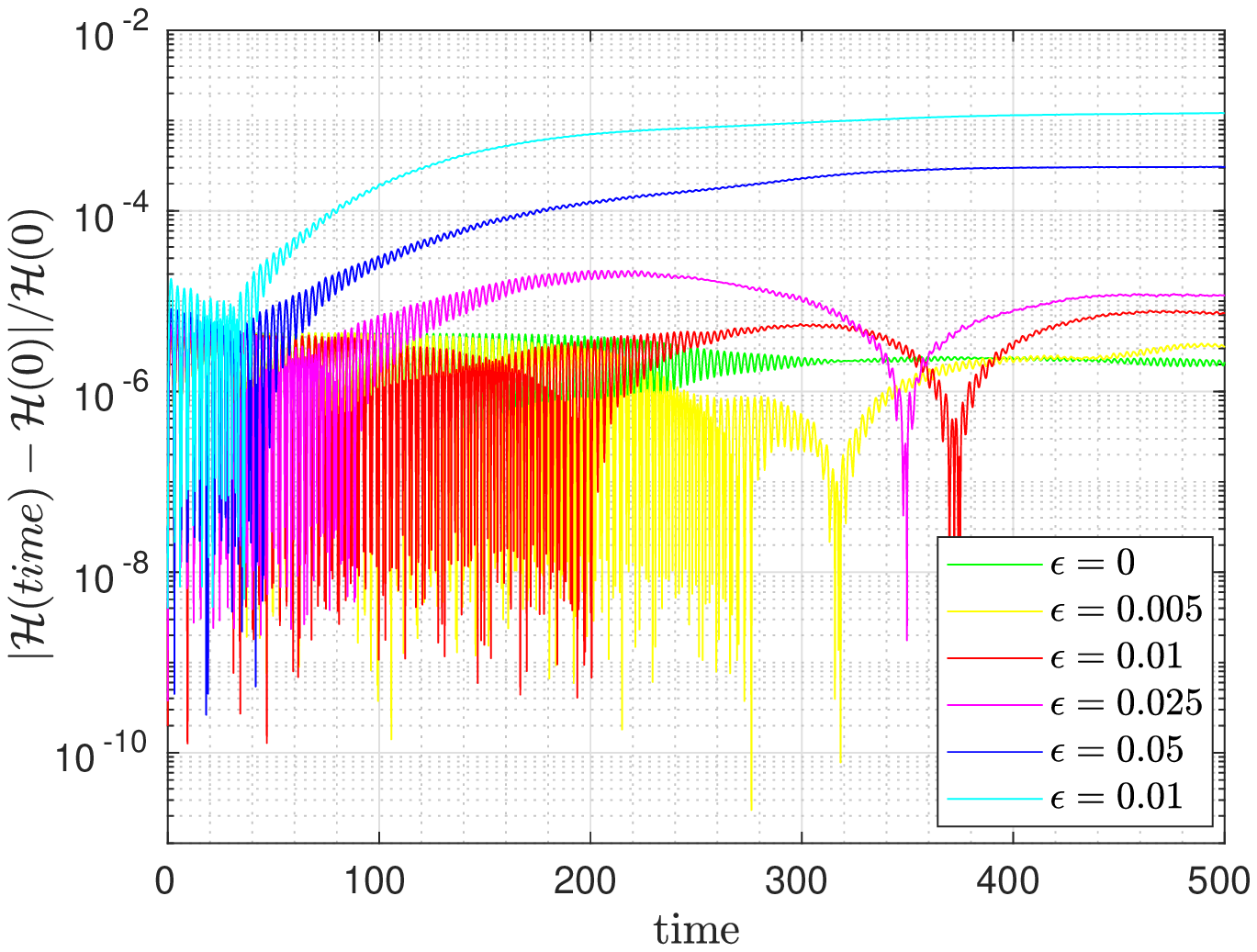}
\caption{Error in total energy}
\label{fig:hsenergy}
 \end{subfigure}
 \caption{Weibel instability on Colella mesh: Hamiltonian splitting with $\Delta t=0.01$ and different values of the distortion parameter $\epsilon$.}
\end{figure}

Now, we look at the conservation properties of our propagators.
Therefore, the time step is set to $\Delta t=0.05$ and the distortion parameter $\epsilon$ to $0.05$ so that all methods run stable. For the fully implicit step in DisGradEC we need on average 4 iterations per time step on the Cartesian grid, 7 on the Orthogonal non-uniform mesh and 8 on the Colella mesh.

\begin{table}[htbp]
\centering
\caption{Weibel instability: Maximum error in Gauss' law and the total energy until time 500 for various integrators with $\Delta t
=0.05$ and $\epsilon=0.05$.}
\begin{tabular}{|c||c||c|c|c|}
\hline
&Method & Cartesian & Orthogonal & Colella  \\
\hline
\multirow{3}{*}{Gauss} &HS & $1.9 \cdot 10^{-11}$ & $5.9 \cdot 10^{-10}$ & $6.8 \cdot 10^{-10}$ \\
\cline{2-5}
&DisGradE & $1.1 \cdot 10^{-6}$ & $ 1.7 \cdot 10^{-6}$&  $ 1.6 \cdot 10^{-6}$\\
\cline{2-5}
&DisGradEC & $3.8 \cdot 10^{-13}$ & $6.4 \cdot 10^{-10}$ & $ 8.6\cdot 10^{-10}$ \\
\hline
\multirow{3}{*} {Energy} &HS  & $1.1 \cdot 10^{-4}$ & $1.8 \cdot 10^{-4}$ & $1.6 \cdot 10^{-3}$\\
\cline{2-5}  
&DisGradE & $3.2 \cdot 10^{-10}$ & $ 1.4 \cdot 10^{-10}$ & $ 4.2 \cdot 10^{-10}$\\
\cline{2-5}
&DisGradEC & $6.0\cdot 10^{-12}$& $1.6 \cdot 10^{-10}$ & $ 4.2\cdot 10^{-10}$\\
\hline
\end{tabular}
\label{table:1}
\end{table}
In Table \ref{table:1}, we see the difference between the energy and the charge conserving methods. As expected, the discrete gradient methods  (DisGradE, DisGradCE) conserve the total energy whereas for the Hamiltonian splitting method (HS) the energy is not conserved but the error is bounded as can be seen in Figure \ref{fig:hsenergy}. 
As proven in Section \ref{sec:structure}, the charge conserving discrete gradient method (DisGradCE) and the Hamiltonian splitting method conserve Gauss' law. 
Note that all the conservations are up to the tolerance of the solver times the condition number of the mass matrices.

\subsection{Strong Landau damping}
We also study the electrostatic Landau damping with initial distribution
\begin{align*}
f(\xb,\vb,t=0)= \left(1+\alpha \cos( \mathbf{k} \cdot \xb)\right)\frac{1}{(2\pi)^\frac{3}{2} v_{th1}^3} \exp \left( -\frac{1}{2} \left( \frac{\vb^2}{v_{th1}^2}\right) \right ),  \xb \in [0,L]^3, \vb \in \R^3.
\end{align*}
The magnetic field is set to zero and for the electrostatic setting Farady's equation is excluded. The electric field at initial time, $\Eb(\xb,t=0)$, is calculated from Poisson's equation.
We choose the parameter as $v_{th1}=1, \mathbf{k}=(0.5,0,0)^\top, \alpha=0.5, L=\frac{2 \pi}{0.5}$.
For the numerical resolution, we take $1,600,000$ particles, $32\times4\times2$ grid cells, spline degrees $(3,2,1)$, time step of $\Delta t=0.05$ and for the  iterative solver a tolerance of $10^{-13}$, for the non-linear iteration in DisGradEC a tolerance of $10^{-12}$.
These parameters extend the 1D2V settings in \cite{kraus2016gempic} to the 3D3V phase space.
\begin{figure}
\includegraphics[width=0.5\textwidth]{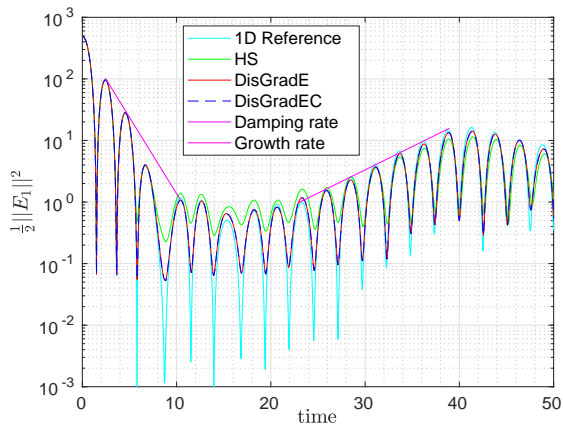}
\centering
\caption{Landau Damping on Colella mesh: First component of electric energy with time step $\Delta t=0.05$ and distortion parameter $\epsilon=0.1$ for the transformation.}
\label{fig:landauE1}
\end{figure}

Figure \ref{fig:landauE1} shows the electric energy for various integrators on the Colella mesh together with a 1D reference run on the Cartesian grid. All propagators yield similar results which fit with the damping and growth rate obtained from the 1D test case. However, the symplectic Hamiltonian splitting shows worse results than the implicit schemes.

\begin{table}[htbp]
\centering
\caption{Landau Damping: Maximum error in Gauss' law and the total energy until time 500 for various integrators with $\Delta t
=0.05$ and $\epsilon=0.1$.}
\begin{tabular}{|c||c|c||c|c|}
\hline
&\multicolumn{2}{|c||} {Gauss} & \multicolumn{2}{|c|} {Energy}\\
\hline
Method & Cartesian & Colella & Cartesian & Colella\\
\hline
HS & $3.1\cdot 10^{-13}$ &  $3.1\cdot 10^{-11}$& $1.0 \cdot 10^{-4}$ & $10.0 \cdot 10^{-3}$\\
\hline
DisGradE & $7.4\cdot 10^{-3}$ & $2.1\cdot 10^{-2}$ & $3.1 \cdot 10^{-14}$ & $2.7 \cdot 10^{-14}$\\
\hline
DisGradEC & $2.5 \cdot 10^{-13}$ & $3.5 \cdot 10^{-11}$ & $1.4 \cdot 10^{-13}$ &  $1.3\cdot 10^{-14}$\\
\hline
\end{tabular}
\label{table:2}
\end{table}
From Table \ref{table:2} it becomes obvious that the constructed conservation properties are satisfied numerically.

\section{Conclusions}\label{sec:conclusions}
We have derived a geometric particle-in-cell method in curvilinear geometry based on a discretisation of the fields with finite element exterior calculus and a hybrid particle pusher. Our formulation yields a semi-discrete Poisson system that satisfies the Jacobi identity. For the discretisation in time, we have considered both a variational integrator as well as energy-conserving time stepping schemes based on the discrete gradient method and an antisymmetric splitting of the Poisson matrix. 

In order to investigate the influence of the coordinate transformation, we have restricted ourselves to test problems with periodic boundaries, where exact conservation is achieved. However, in more realistic scenarios real boundary conditions apply and the energy balance at the boundaries has to be considered. This  will be the topic of forthcoming work.

\section{Acknowledgements}
The authors thank Douglas Arnold, Elena Celledoni, Roman Hatzky, Philip J. Morrison, Nils Mosch{\"u}ring, Brynjulf Owren and Hendrik Speleers for helpful discussions.

\end{document}